\documentclass[11pt]{amsart}

\usepackage[utf8]{inputenc} 	
\usepackage[T1]{fontenc}
\usepackage{lmodern}
\usepackage{yhmath}

\usepackage{amsthm,amsmath,amssymb,amsbsy,bbm,mathrsfs,supertabular,
eurosym,graphicx,enumitem,xcolor}
\usepackage{mathtools}

\newtheoremstyle{BBstyle0}  {}{}{\itshape}{}{\bfseries}{}{6pt}{}
\newtheoremstyle{BBstyle1}  {3pt}{3pt}{\rmfamily}{}{\itshape}{: }{3pt}{}
\newtheoremstyle{BBstyle2}  {3pt}{3pt}{\itshape}{}{\bfseries\large}{}{0pt}{}
\newtheoremstyle{BBstyle3}  {}{}{\itshape}{}{\bfseries}{: }{3pt}{}
\newtheoremstyle{BBstyle4}  {}{}{\rmfamily}{}{\bfseries}{}{6pt}{}
  
\usepackage[normalem]{ulem} 
\usepackage[authoryear]{natbib}

\newtheorem{thm}{Theorem}
\newtheorem{lem}{Lemma}
\newtheorem{prop}{Proposition}
\newtheorem{df}{Definition}
\newtheorem{cor}{Corollary}

\theoremstyle{definition}
\newtheorem{exa}{Example}

\usepackage[english]{babel}

\usepackage{enumitem}
\usepackage{hyperref}  
\newcommand{\pa}[1]{\left({#1}\right)}
\newcommand{\norm}[1]{\left\|{#1}\right\|}
\newcommand{\cro}[1]{\left[{#1}\right]}
\newcommand{\ab}[1]{\left|{#1}\right|}
\newcommand{\ac}[1]{\left\{{#1}\right\}}

\newcommand{\Var}{\mathop{\rm Var}\nolimits}

\newcommand{\dfleche}[1]{\,\displaystyle{\mathop{\longrightarrow}_{#1}}\,}

\newcommand{\CV}[1]{\dfleche{#1}}

\newcommand{\E}{{\mathbb{E}}}
 
\renewcommand{\L}{{\mathbb{L}}}
\newcommand{\N}{{\mathbb{N}}}

\newcommand{\R}{{\mathbb{R}}}

\newcommand{\sS}{{\mathscr{S}}}

\DeclareMathAlphabet{\mathscrbf}{OMS}{mdugm}{b}{n}
\newcommand{\cA}{{\mathcal{A}}}

\newcommand{\cC}{{\mathcal{C}}}

\newcommand{\cE}{{\mathcal{E}}}
\newcommand{\cF}{{\mathcal{F}}}
\newcommand{\cG}{{\mathcal{G}}} 
\newcommand{\cH}{{\mathcal{H}}}
\newcommand{\cI}{{\mathcal{I}}}
\newcommand{\cJ}{{\mathcal{J}}}

\newcommand{\cM}{{\mathcal{M}}}

\newcommand{\cO}{{\mathcal{O}}}

\newcommand{\cS}{{\mathcal{S}}}

\newcommand{\cX}{{\mathcal{X}}}
\newcommand{\cY}{{\mathcal{Y}}} 
\newcommand{\cZ}{{\mathcal{Z}}}

\newcommand{\gv}{{\mathbf{v}}}

\newcommand{\gx}{{\mathbf{x}}}

\newcommand{\gT}{{\mathbf{T}}}

\newcommand{\bs}[1]{\boldsymbol{#1}}

\newcommand{\bsZ}{{\bs{Z}}}
\newcommand{\frc}{{\mathfrak{c}}}

\newcommand{\frm}{{\mathfrak{m}}}

\newlist{lista}{enumerate}{1}
\setlist[lista,1]{label=\alph*),ref=\alph*)}

\newlist{listi}{enumerate}{1}
\setlist[listi,1]{label=(\roman*),ref=(\roman*),align=left}

\newcommand{\eref}[1]{(\ref{#1})}
\renewcommand{\ge}{\geqslant}
\renewcommand{\le}{\leqslant}
\newcommand{\1}{1\hskip-2.6pt{\rm l}}

\newcommand{\scal}[2]{\langle #1,#2\rangle}

\newcommand{\etc}[1]{#1_1,\ldots,#1_n}

\newcommand{\dps}[1]{\displaystyle{#1}}

\newcommand{\et}{^{\star}}
\def\vide{{\varnothing}}
\newcommand{\eps}{{\varepsilon}}

\def\sgn{{\mathrm{\:sgn}}}

\DeclarePairedDelimiter\ceil{\lceil}{\rceil}
\DeclarePairedDelimiter\floor{\lfloor}{\rfloor}
\newcommand{\PES}[1]{\ceil*{#1}}
\newcommand{\PEI}[1]{\floor*{#1}}

\usepackage{pgfplots}
\pgfplotsset{compat=1.18}

\begin{document}
\title[]{Estimating a regression function under possible heteroscedastic and heavy-tailed errors. Application to shape-restricted regression. }
\author{Yannick BARAUD} 
\address{University of Luxembourg\\
DMATH\\
Maison du nombre\\
6 avenue de la Fonte\\
L-4364 Esch-sur-Alzette\\
Grand Duchy of Luxembourg.}
\email{yannick.baraud@uni.lu}

\author{Guillaume MAILLARD}
\address{ENSAI\\
Campus de Ker Lann\\
51 Rue Blaise Pascal\\
BP 37203\\
35172 BRUZ Cedex.}
\email{guillaume.maillard@ensai.fr}

\date{\today}
\begin{abstract}
We consider a regression framework where the design points are deterministic and the errors possibly non-i.i.d.\ and heavy-tailed (with a moment of order $p$ in $[1,2]$). Given a class of candidate regression functions, we propose a surrogate for the classical least squares estimator (LSE). For this new estimator, we establish a nonasymptotic risk bound with respect to the absolute loss  which takes the form of an oracle type inequality.  This inequality shows that our estimator possesses natural adaptation properties with respect to some elements of the class. When this class consists of monotone functions or convex functions on an interval, these adaptation properties are similar to those established in the literature for the LSE. However, unlike the LSE, we prove that our estimator remains stable with respect to a possible heteroscedasticity of the errors and may even converge at a parametric rate (up to a logarithmic factor) when the LSE is not even consistent. We illustrate the performance of this new estimator over classes of regression functions that satisfy a shape constraint: piecewise monotone,  piecewise convex/concave,  among other examples. The paper also contains some approximation results by splines with degrees in $\{0,1\}$ and VC bounds for the dimensions of classes of level sets. These results may be  of independent interest. 
 \end{abstract}

\maketitle

\section{Introduction}
In this paper, we consider the problem of estimating a regression function $f\et$ on some set $\cX$ under a fixed design setting with independent and centred real-valued errors $\xi_{1,n},\ldots,\xi_{n,n}$. More precisely, we choose $n$ deterministic points $x_{1,n},\ldots,x_{n,n}$ in $\cX$ (possibly not all different) and associate each of them with a response $Y_{i,n}$ with values in $\R$ of the following form:
\begin{equation}\label{def-fet}
Y_{i,n}=f\et(x_{i,n})+\xi_{i,n}\quad\text{so that}\quad f\et(x_{i,n})=\E\cro{Y_{i,n}}\quad \text{for all $i\in\{1,\ldots,n\}$.}
\end{equation}
We only assume that the $\xi_{i,n}$ are independent (not necessarily i.i.d.) and that all of them have 
a finite moment of order $p\in[1,2]$ with $\E[\xi_{i,n}]=0$. On the basis of these observations, our aim is to estimate the so-called regression function $f\et$ from $\cX$ to $\R$. Our approach is based on {\em a model $\cF$ for $f\et$}, which means that we choose some family $\cF$ of candidate regression functions and pretend that $f\et$ belongs to $\cF$, although this may not be true. To get a good estimator of $f\et$, it actually suffices that $\cF$ provides a good enough approximation of it. For instance, $\cF$ may be a finite dimensional linear space or the set of all nonincreasing functions on $[0,1]$.

Under the fixed design framework, a classical loss $\ell_{s}$ (with $s\ge1$) between two functions
$f,g$ on $\cX$ is given by
\begin{equation}
\ell_{s}(f,g)=\cro{\frac{1}{n}\sum_{i=1}^{n}\ab{f(x_{i,n})-g(x_{i,n)}}^{s}}^{1/s},
\label{eq-ls-loss}
\end{equation}
the case of $s=2$ being the most common. The corresponding risk of an estimator
$\widehat{f}$ of the true regression function $f\et$ is then
\[
\E\cro{\ell_{s}\pa{f\et,\widehat f}}=\E\left[\pa{\frac{1}{n}\sum_{i=1}^{n}
\ab{f\et(x_{i,n)}-\widehat f(x_{i,n})}^{s}}^{1/s}\right].
\]
As this quantity may involve $\E\cro{|\xi_{i,n}|^{s}}$ for $1\le i\le n$, it is commonly assumed in the literature that these moments are finite when one uses the $\ell_{s}$-loss. The value of $p$ being unknown, but not smaller than one, the moment condition is always satisfied with $s=1$ and we shall therefore measure the loss by
$\ell_{1}$. Throughout the paper, this loss is simply denoted by $\ell$ so that for two functions
$f,g$ on $\cX$, 
\[
\ell(f,g)=\frac{1}{n}\sum_{i=1}^{n}\ab{f(x_{i,n})-g(x_{i,n})}.
\]

Our purpose is to bound $\E\cro{\ell\pa{f\et,\widehat f}}$ for a suitable estimator $\widehat{f}$
whose construction will be described in Section~\ref{sect-CE}. As in the papers Baraud and Birg\'e~\citeyearpar{MR3565484} and Baraud, Halconruy and Maillard~\citeyearpar{Baraud:2022aa}, both in density estimation, this construction is based on a model, namely $\cF$,
and a suitable family of tests. Given $\cF$, one can actually define a nondecreasing sequence $\{\cO(D),1\le D\le n\}$ of subsets of $\cF$ with respect to set inclusion, 
for which the risk of our estimator is uniformly bounded by some quantity $R_{D}$ on each subset $\cO(D)$ with $D\in\{1,\ldots,n\}$.  This bound is increasing 
with $D$ so that the smaller $D$ the better the estimation for $f\et\in\cO(D)$. For a general $f\et$
and a given value of $D$, the risk of our estimator is the sum of $R_{D}$ and an approximation
term $3\ell(f\et,\cO(D))$, where, for $\cG\subset\cF$, $\ell(f,\cG)$ denotes the quantity
$\inf_{g\in\cG}\ell(f,g)$, with the convention that it takes the value $+\infty$ when $\cG$ is empty.
As to $R_{D}$, it is proportional to $\sigma_{p}(\gx)$ with $\gx=(x_{1,n},\ldots,x_{n,n})$ provided
that
\begin{equation}\label{eq-momentp}
\sigma_{p}^{p}(\gx)=\frac{1}{n}\sum_{i=1}^{n}\E\cro{\ab{\xi_{i,n}}^{p}}=\frac{1}{n}\sum_{i=1}^{n}\E\cro{\ab{Y_{i,n}-\E\cro{Y_{i,n}}}^{p}}<+\infty,
\end{equation}
while $D$ plays a role which is similar to that of the dimension of the parameter space in Gaussian linear regression. Indeed, for $p=2$ we obtain that
\[
\E\cro{\ell\pa{f\et,\widehat f}}\le3\ell(f\et,\cO(D))+\kappa\sigma_{2}(\gx)\sqrt{\frac{D}{n}},
\]
for some universal positive constant $\kappa$. More generally, for every $D\in\{1,\ldots,n\}$, $p\in[1,2]$
and every regression function $f\et$ on $\cX$,
\begin{equation}
\E\cro{\ell\pa{f\et,\widehat f}}\le 3\ell\pa{f\et,\cO(D)}+ \kappa\cro{\sigma_{p}(\gx)\pa{\frac{D}{n}}^{1-\frac{1}{p}}},
\label{eq-risk-z}
\end{equation}
this bound being also true if $\sigma_{p}(\gx)=+\infty$.
Since it is valid for every $D$ and $p$, the estimator automatically adapts to the optimal values of $D$ and $p$, although neither the values $\sigma_{p}(\gx)$ nor those of $p$ for which these quantities are finite are assumed to be known. The case $p=1$ is the minimal requirement that ensures the existence of the regression function in (\ref{def-fet}). The bound also implies that the estimator is robust in the sense that if $f\et$ does not belong to $\cO(D)$ for a given value of $D$, the extra incurred loss is of the form $3\ell(f\et,\cO(D))$. Moreover $\sigma_{p}(\gx)$
only depends on the mean of the $p$-th absolute moments of the $\xi_{i,n}$, not their maximum. It may therefore be small even if one of these moments is quite large. For instance, an outlier with a large value of $\E\cro{\ab{\xi_{i,n}}^{p}}$ while all the others are small will not perturbate the result too much, which can be viewed as another robustness property of the method.

The problem of estimating a regression function under a fixed design has received a great deal of attention in statistics with a particular emphasis on the least squares method. Actually, we are not aware of many alternatives to the Least Squares Estimator (LSE for short), particularly in a heteroscedastic setting. A notable exception is Efromovich and Pinsker~\citeyearpar{Efromovich1996}. Under the assumption that the errors are Gaussian, the authors designed an estimator which is both adaptive and sharp minimax over classes of Sobolev spaces. 

It is beyond the scope of this paper to provide a complete review of the use of least squares in statistics. As we shall mainly apply our procedure for the purpose of estimating a regression function with a given shape, we restrict our references to papers dealing with applications to such models.

The case of the class $\cF$ of monotone functions on an interval $\cX\subset\R$ has been the most studied in the literature. Zhang~\citeyearpar{Zhang2002}
studied the performance of the LSE on $\cF$ when the errors $\xi_{i,n}$ are independent (but not necessarily i.i.d.) and their variances are not larger than $\sigma^{2}$. He proved that the $\ell_{s}$-risk of the LSE, $s\ge 1$, for estimating a nondecreasing regression function with variation $V$ on $\cX$ is not larger than
\[
C(s)\cro{\pa{\frac{V\sigma^{2}}{n}}^{1/3}+\sigma r_{n}}\quad \text{where}\quad r_{n}=
\begin{cases}
1/\sqrt{n} & \text{for $s\in [1,2)$}\\
\sqrt{(\log n)/n}& \text{for $s=2$}
\end{cases}
\]
and $C(s)$ is a positive constant depending on $s$ only.
When $s=2$, Chatterjee {\em et al}~\citeyearpar{chatterjee2015} proved that the term $(V\sigma^{2}/n)^{1/3}$ cannot be improved in general, at least when the errors are Gaussian and $n$ grows to infinity while $V$ remains fixed. From a nonasymptotic point of view, the second term $\sigma r_{n}$ becomes dominant when $V$ is sufficiently small, that is, when $f\et$ is almost constant. The fact that the $\ell_{2}$-risk of the LSE is not uniform over the class and may depend on the structure of the target regression function $f\et$ is often called {\em adaptation}. This terminology refers here to an adaptation with respect to {\em a shape} rather than {\em a smoothness} as classically studied in the literature. This property was investigated further in Chatterjee {\em et al}~\citeyearpar{chatterjee2015}. For i.i.d.\ Gaussian errors, they proved that the $\ell_{2}$-risk of the LSE is not larger than $\sigma\sqrt{6(k\log n)/n}$ when the regression function is nondecreasing and piecewise constant on a partition of $\cX$ into $k\in\{1,\ldots,n\}$ intervals. 
Besides, they showed that this risk bound remains of the same order for any regression function $f\et$ that lies close enough to an element of this form. This can be viewed as a {\em robustness} property. It is characterised by an oracle-type inequality which is of the same flavour as those established for penalised least squares estimators in the literature (see Birg\'e and Massart~\citeyearpar{MR1848946} and the references therein). Bellec~\citeyearpar{Bellec2018} obtained a sharp oracle-type inequality for the $\ell_{2}$-risk of the isotonic LSE when the errors are i.i.d.\ Gaussian. When the errors are only assumed to be independent with finite variances, an inequality of the same type was established by Guntuboyina and Sen~\citeyearpar{Guntuboyina2018}. They also extended it to $\ell_{s}$-risks with $s\in [1,2)$ under the assumption of normality of the errors. 

The $\ell_{2}$-risk of the LSE has also been studied for the problem of estimating regression functions satisfying other shape constraints than monotonicity. Chatterjee~\citeyearpar{Chatterjee2016} considered the class $\cF$ of convex functions on an interval. When the design points are equispaced and the regression function is affine, the author proves that the $\ell_{2}$-risk of the LSE becomes of order $\sqrt{\log n/n}$. The LSE's adaptation properties with respect to convex piecewise linear regression functions are established in Bellec~\citeyearpar{Bellec2018} where the author obtained again sharp oracle inequalities. 

A common feature of these classes of functions (nondecreasing or convex) lies in their convexity. A more general study of the properties of the LSE on convex sets of functions can be found in Chatterjee~\citeyearpar{MR3269982}. Less is known about the performance of the LSE on non-convex sets. A notable exception is the class of unimodal functions for which Bellec~\citeyearpar{Bellec2018} (see also Chatterjee and Lafferty~\citeyearpar{Chatterjee2019}) established similar oracle-type inequalities with respect to piecewise constant functions as in the isotonic case. Minami~\citeyearpar{Minami2020} considered the larger class of regression functions which are piecewise monotone but the oracle-type inequality established there requires some constraints on the minimal length of the intervals of the partition. Finally, Feng {\em et al}~\citeyearpar{S-shaped2022} studied the global rates and adaptation properties of the LSE for estimating a regression function with an $S$-shape on $[0,1]$, that is, which are first convex and then concave, when the location of the inflexion point is unknown. 

The large amount of literature on the LSE seems to indicate the difficulty of establishing a fairly general result on its performance and the previous review also sketches out some of its limitations. All the results we have mentioned rely on the assumptions that the errors are, at least, square integrable, although the very definition of the regression function only requires their integrability. They are furthermore mostly restricted to the homoscedastic framework or provide risk bounds depending on the maximum of these variances.

The aim of this paper is to propose a surrogate for the celebrated least squares estimator. As we shall see, we establish its performance in a fairly general framework. Our construction relies on robust tests between two candidate regression functions which are, to our knowledge, new. Nevertheless, this approach enjoys some similarities with that previously adopted in Baraud, Halconruy and Maillard~\citeyearpar{Baraud:2022aa} for estimating a density with respect to the total variation distance. 

Our main result, namely the risk bound (\ref{eq-risk-z}), holds under quite weak assumptions which differ substantially from the usual ones in the classical regression framework:
\begin{listi}
\item The random variables $\xi_{i,n}$ have a moment of unknown order $p\in[1,2]$, therefore possibly
no finite variance.
\item The independent variables $\xi_{i,n}$ are not necessarily i.i.d.\ so that the values of the moments
$\E\cro{|\xi|^{p}_{i,n}}$ for $1\le i\le n$ may be quite different.
\item We do not assume that $f\et\in\cF$.
\end{listi}
This versatility allows us to apply our approach to various regression models (including those based on a shape constraint) when the errors are possibly non-i.i.d.\ and strongly heavy tailed (with a moment of order $p\in [1,2)$ only). In particular, our estimator can still be consistent when the errors are heteroscedastic and the variances of a few of them larger than the sample size.  Our examples include classes of regression functions which are piecewise monotone, for example unimodal, or piecewise convex/concave, for example with an $S$-shape. In this latter case, we also establish risk bounds for possibly nonequispaced design points, a problem which, to our knowledge, has not been tackled in the literature. Shape-restricted single index models are also considered. We are not aware of another approach that allows the statistician to obtain an estimator which remains stable in a possibly heteroscedastic context. This phenomenon is due to the fact that its risk bound depends on the average {of} the variances and not the maximum of  {these}. For illustration, we shall see that our estimator may take advantage of this property to converge at rate $(\log n)/\sqrt{n}$ in a situation where {the} LSE is not even consistent. This property also makes it useful for estimation in the Logistic and Poisson regression frameworks which are heteroscedastic by nature. It offers thus an alternative to the $\rho$-estimators that were studied in Baraud and Chen~\citeyearpar{MR4725162} (see also the references therein) in these settings.

The paper is organised as follows. The statistical framework and the construction of the estimator are described in Section~\ref{sect-I} which also contains a comparison {between our estimator and the LSE} on some simple examples. Section~\ref{sect-CEMR} is devoted to the main result while Section~\ref{sect-FATH} {deals with} our applications to the models of regression functions we have {previously described}. The reader can also find in Section~\ref{sect-FATH} some $\ell_{1}$-approximation bounds that can be achieved by using splines. These results can be of independent interest. The technical calculation of some VC-bounds that we use throughout the paper are postponed to Section~\ref{sect-BD}. The proof of our main result can be found in {Section~\ref{sect-pfTh1}} and the other proofs in Section~\ref{Sect-OP}.

Throughout the paper, we shall use the following notation and conventions. The cardinality of a set $A$ is denoted by $|A|$ and we shall say that $A\subset \R$ is a {\em nontrivial interval} if it is an interval with {positive length}. Given a nonnegative number $x\in\R$, $\PES{x}$ is the {smallest integer which is at least as large as $x$} while $\PEI{x}$ is the largest integer not larger than $x$. We write $x\wedge y$ and $x\vee y$ respectively for the minimum and maximum between the two numbers $x$ and $y$ in $\R$ and
\[
\sgn(x)=\1_{x>0}-\1_{x<0}
=
\begin{cases}
1& \text{ if $x>0$},\\
0 &\text{ if $x=0$},\\
-1&\text{ if $x<0$}.\\
\end{cases}
\quad \text{while}\quad \sgn(x,y)=\sgn(x-y).
\]

Finally, we use the conventions $\sum_{\vide}=0$ and $\inf_{\vide}=+\infty$.


\section{The statistical framework and the construction of the estimator}\label{sect-I}

\subsection{The setting and first examples}
As already mentioned in the introduction, our aim is to estimate the regression function $f\et$ defined in \eref{def-fet} from the observation of $\bsZ=\{(x_{1,n},Y_{1,n}),\ldots,(x_{n,n},Y_{n,n})\}$. Given a family $\cF$ of candidate functions on $\cX$, we evaluate the performance of an estimator $\widehat f$ with values in $\cF$ by means of the $\ell_{1}$-loss. We recall that neither the values $\sigma_{p}(\gx)$ nor those of $p$ for which these quantities are finite are assumed to be known. 

 Let us now turn to some examples that illustrate our statistical framework.
\begin{exa}\label{exa0}
The observations $Z_{i}=(x_{i,n},Y_{i,n})$ are drawn from the model
\[
Y_{i,n}=f\et(x_{i,n})+\tau_{i,n}\xi_{i}\quad \text{for $i\in\{1,\ldots,n\}$}
\]
where the $\tau_{i,n}$ are nonnegative numbers and the $\xi_{i}$  i.i.d.\ centred random variables with a $p$-th absolute moment $\E\cro{|\xi_{1}|^{p}}=1$ for some unknown $p\ge 1$. Then, our moment assumption is satisfied with $\sigma_{p}^{p}(\gx)=n^{-1}\sum_{i=1}^{n}\tau_{i,n}^{p}$.
\end{exa}

\begin{exa}\label{exa2}
Given $\beta>0$, let $q_{\beta}$ be the function defined on  $[0,+\infty)$ by
\[
q_{\beta}: u\mapsto \frac{\pa{2\log u}^{\beta-1}}{G(\beta)u^{3}}\1_{[1,+\infty)}(u)
\]
where  $G$ denotes the gamma function, that is $G(\beta)=\int_{0}^{+\infty}t^{\beta-1}e^{-t}dt$. We consider the regression framework
\[
Y_{i,n}=f\et(x_{i,n})+\xi_{i}\quad \text{for $i\in\{1,\ldots,n\}$}
\]
where the  $\xi_{i}$ are i.i.d.\ centred random variables with density $x\mapsto q_{\beta}(|x|)$. In this case, our moment assumption is satisfied for any $p\in[1,2)$ while $\E\cro{|\xi_{1}|^{2}}=+\infty$.
\end{exa}

\begin{exa}\label{exai3}
For each $i\in\{1,\ldots,n\}$, $Y_{i,n}$ is distributed as a Bernoulli random variable with mean $f\et(x_{i,n})\in [0,1]$. Then our assumptions are satisfied with $p=2$ and
\[
\sigma_{2}^{2}(\gx)=\frac{1}{n}\sum_{i=1}^{n}f\et(x_{i,n})\pa{1-f\et(x_{i,n})}.
\]
In particular, when $\cX=[0,1]$, $x_{i,n}=(i-1/2)/n$ for $i\in\{1,\ldots,n\}$ and $f\et$ is continuous on $[0,1]$,
\[
\sigma_{2}^{2}(\gx)=\sigma_{2}^{2}(\gx,n)\CV{n\to +\infty} \sigma_{2}(f\et)=\int_{0}^{1}f\et(x)\pa{1-f\et(x)}dx.
\]
%
%
\end{exa}

\begin{exa}\label{exai4}
For each $i\in\{1,\ldots,n\}$, $Y_{i,n}$ is a Poisson random variable with mean $f\et(x_{i,n})>0$. Then our assumptions are satisfied with $p=2$ and
\[
\sigma_{2}^{2}(\gx)=\frac{1}{n}\sum_{i=1}^{n}f\et(x_{i,n}).
\]
In particular, when $\cX=[0,1]$, $x_{i,n}=(i-1/2)/n$ for $i\in\{1,\ldots,n\}$ and $f\et$ is continuous on $[0,1]$,
\[
\sigma_{2}^{2}(\gx)=\sigma_{2}^{2}(\gx,n)\CV{n\to +\infty} \sigma_{2}(f\et)=\int_{0}^{1}f\et(x)dx.
\]

\end{exa}
\subsection{Construction of the estimator}\label{sect-CE}
Our construction is based on the test statistic $T(\bsZ,f,g)$ which is defined for  $f,g\in \cF$ by
\[
T(\bsZ,f,g)=\sum_{i=1}^{n}\pa{f(x_{i,n})-Y_{i,n}}\sgn(f,g,x_{i,n}),
\]
where $\sgn(f,g,x)=\sgn(f(x),g(x))$ for all $x\in\cX$.
The statistic $T(\bsZ,f,g)$ possesses the following properties.
\begin{lem}\label{lem01}
For every $f,g\in \cF$,
\begin{equation}\label{eq-lem01}
n\cro{\ell(f\et,f)-2\ell(f\et,g)}\le n\cro{\ell(f,g)-\ell(f\et,g)}\le \E\cro{T(\bsZ,f,g)}\le n\ell(f\et,f).
\end{equation}
\end{lem}
In particular, if $\ell(f\et,f)$ is larger than $2\ell(f\et,g)$, the expectation of our test statistic $T(\bsZ,f,g)$ is necessarily large while it is small when $\ell(f\et,f)$ is small. 
By taking the supremum with respect to $g\in\cF$ in \eref{eq-lem01}, we obtain that
\[
n\cro{\ell(f\et,f)-2\ell(f\et,\cF)}\le \sup_{g\in\cF}\E\cro{T(\bsZ,f,g)}\le n\ell(f\et,f)\quad \text{for every $f\in\cF$.}
\]
An element $\overline f$ which minimises $\gT:f\mapsto \sup_{g\in\cF}\E\cro{T(\bsZ,f,g)}$ over $\cF$  satisfies thus for every $f\in\cF$,
\begin{align*}
\ell(f\et,\overline f)&\le  n^{-1}\sup_{g\in\cF}\E\cro{T(\bsZ,\overline f,g)}+2\ell(f\et,\cF)\\
&\le n^{-1}\sup_{g\in\cF}\E\cro{T(\bsZ,f,g)}+2\ell(f\et,\cF)\le \ell(f\et,f)+2\ell(f\et,\cF).
\end{align*}
Since $f$ can be chosen arbitrarily in $\cF$, we derive that $\ell(f\et,\overline f)\le 3\ell(f\et,\cF)$. This means that the minimiser $\overline f$ of $\gT$ almost provides a best approximation of $f\et$ in $\cF$. The heuristic that supports our approach lies on the fact that if $T(\bsZ,f,g)$ remains close enough to its expectation uniformly over $f,g\in\cF$, we can substitute $T(\bsZ,f,g)$ for $ \E\cro{T(\bsZ,f,g)}$. By doing so and setting
\begin{equation}\label{eq-defZf}
T(\bsZ,f)=\sup_{g\in \cF}T(\bsZ,f,g)\quad \text{for every $f\in\cF$,}
\end{equation}
we define our estimator $\widehat f$ of $f\et$ as an (approximate) minimiser of $f\mapsto T(\bsZ,f)$ over $\cF$, that is, an element of the (random) set
\begin{equation}\label{eq-defEk}
\cE_{\frc}(\bsZ,\cF)=\ac{f\in\cF,\; T(\bsZ,f)\le \inf_{g\in\cF}T(\bsZ,g)+\frc}
\end{equation}
where $\frc$ denotes some (small enough) nonnegative numerical constant.


\begin{proof}[Proof of Lemma~\ref{lem01}]
Note that
\begin{equation}
\cro{f(x_{i,n})-g(x_{i,n})}\sgn(f,g,x_{i,n})=\ab{f(x_{i,n})-g(x_{i,n})}.
\label{eq-LB1}
\end{equation}
For each  $i\in\{1,\ldots,n\}$,
\begin{align*}
\E\cro{\pa{f(x_{i,n})-Y_{i,n}} \sgn(f,g,x_{i,n})}&=\cro{f(x_{i,n})-f_{i}\et(x_{i,n})} \sgn(f,g,x_{i,n})
\le \ab{f(x_{i,n})-f_{i}\et(x_{i,n})}.\end{align*}
We derive the right-hand side of \eref{eq-lem01} by summing over $i$. Let us now turn to the proof of the left-hand sides.
\begin{align*}
\lefteqn{\E\cro{\pa{f(x_{i,n})-Y_{i,n}} \sgn(f,g,x_{i,n})}}\hspace{15mm}\\
&=\pa{f(x_{i,n})-g(x_{i,n})} \sgn(f,g,x_{i,n})-\pa{f_{i}\et(x_{i,n})-g(x_{i,n})} \sgn(f,g,x_{i,n})\\
&\ge \ab{f(x_{i,n})-g(x_{i,n})}-\ab{f_{i}\et(x_{i,n})-g(x_{i,n})}\\&\ge \ab{f(x_{i,n})-f_{i}\et(x_{i,n})}-2\ab{f_{i}\et(x_{i,n})-g(x_{i,n})},
\end{align*}
Where we used (\ref{eq-LB1}). We conclude by summing over $i$.
\end{proof}


We end this section with a simple illustrative example.
\begin{exa}\label{exa1}
Let $\cX$ be a bounded interval of $\R$, $\cI$ a partition of $\cX$ into $D$ intervals which contain at least one design point $x_{i,n}$ and $\cF$ the class of functions which are constant on each interval $I\in\cI$. We denote  by $f_{I}$ the value of a function $f\in\cF$ on $I\in\cI$, $N(I)$ the number of $x_{i,n}$ that fall in $I$, $Y_{I}$ the average of the $Y_{i,n}$ for {$x_{i,n}\in I$}, that is, $Y_{I}={\sum_{i=1}^{n}\1_{I}(x_{i,n})}Y_{i,n}/N(I)$. 
With this notation, for $f,g\in\cF$
\begin{align*}
T(\bsZ,f,g)&=\sum_{i=1}^{n}\pa{f(x_{i,n})-Y_{i,n}}\sgn(f,g,x_{i,n})\\
&=\sum_{I\in \cI}
\cro{{\sum_{i=1}^{n}\1_{I}(x_{i,n})}\pa{f(x_{i,n})-Y_{i,n}}\sgn(f,g,x_{i,n})}\\
&=\sum_{I\in \cI}N(I)\pa{f_{I}-Y_{I}}\sgn(f_{I},g_{I})\le \sum_{I\in \cI}N(I)\ab{f_{I}-Y_{I}}.
\end{align*}
Since equality holds for $g=\widehat f=\sum_{I\in\cI}Y_{I}\1_{I}$,
\[
T(\bsZ,f)=\sup_{g\in\cF}T(\bsZ,f,g)=\sum_{I\in \cI}N(I)\ab{f_{I}-Y_{I}},
\]
and the minimum of $f\mapsto T(\bsZ,f)$ is achieved for $f=\widehat f$. Our estimation procedure leads thus to the classical regressogram based on the partition $\cI$. In particular, our estimator is different from the minimiser over $\cF$ of the least absolute values
\[
f\mapsto \sum_{i=1}^{n}\ab{Y_{i,n}-f(x_{i,n})}=\sum_{I\in\cI}\,{\sum_{i=1}^{n}\1_{I}(x_{i,n})}\ab{Y_{i,n}-f_{I}}
\]
which returns the medians of the values {$\{Y_{i,n}\text{ for }x_{i,n}\in I\}$} on the intervals $I\in\cI$ and not their means.
\end{exa}
%


\subsection{Our estimator versus the LSE  on an example}\label{sect-LSvO}
Let us consider the regression framework
\begin{equation}\label{eq-weVersLS}
Y_{i,n}=a\et f_{0}\pa{x_{i,n}}+\xi_{i,n}\quad \text{for $i\in\{1,\ldots,n\}$}
\end{equation}
where $x_{i,n}=i/n$ for $i\in\{1,\ldots,n\}$, $f_{0}$ is a known function in $\L_{1}((0,1],dx)\cap \L_{2}((0,1],dx)$ 
and $\xi_{1,n},\ldots,\xi_{n,n}$ independent centred random variables. We shall assume that 
\[
0<\frac{1}{n}\sum_{i=1}^{n}|f_{0}(x_{i})|\CV{n\to +\infty}\norm{f_{0}}_{1}\quad \text{and}\quad \frac{1}{n}\sum_{i=1}^{n}f_{0}^{2}(x_{i})\CV{n\to +\infty}\norm{f_{0}}_{2}^{2},
\]
where $\norm{\cdot}_{1}$ and $\norm{\cdot}_{2}$ denote the $\L_{1}((0,1],dx)$ and $\L_{2}((0,1],dx)$ norms respectively. Our aim is to estimate the unknown parameter $a\et$.
The LSE on the one-dimensional linear space $\cF=\{af_{0},\; a\in\R\}$ is
\begin{equation}\label{eq-LS}
\widetilde a_{n}=\frac{\sum_{i=1}^{n}Y_{i,n}f_{0}(x_{i,n})}{\sum_{i=1}^{n}f_{0}^{2}(x_{i,n})}=a\et+\frac{\sum_{i=1}^{n}\xi_{i,n}f_{0}(x_{i,n})}{\sum_{i=1}^{n}f_{0}^{2}(x_{i,n})}
\end{equation}
while we shall prove below that our estimator is given by
\begin{equation}\label{eq-OurEst}
\widehat a_{n}=\frac{\sum_{i=1}^{n}Y_{i,n}\sgn(f_{0}(x_{i,n}))}{\sum_{i=1}^{n}\ab{f_{0}(x_{i,n})}}=a\et+\frac{\sum_{i=1}^{n}\xi_{i,n}\sgn(f_{0}(x_{i,n}))}{\sum_{i=1}^{n}\ab{f_{0}(x_{i,n})}}.
\end{equation}
In the ideal situation where the $\xi_{i,n}$ are i.i.d.\ centred with a second moment, say equal to 1, the quadratic risks $ R_{n}^{2}(\widetilde  a_{n}), R_{n}^{2}(\widehat a_{n})$ of $\widetilde a_{n}$ and $\widehat a_{n}$ satisfy
\[
nR_{n}^{2}(\widetilde  a_{n})=\pa{\frac{1}{n}\sum_{i=1}^{n}f_{0}^{2}(x_{i,n})}^{-1}\CV{n\to +\infty}\norm{f_{0}}_{2}^{-2}\in (0,+\infty)
\]
and
\[
nR_{n}^{2}(\widehat a_{n})=\pa{\frac{1}{n}\sum_{i=1}^{n}|f_{0}(x_{i,n})|}^{-2}\CV{n\to +\infty}\norm{f_{0}}_{1}^{-2}\in (0,+\infty).
\]
Both estimators converge thus at rate $1/\sqrt{n}$ and it follows from the Cauchy-Schwarz inequality
that the quadratic risk of $\widetilde  a_{n}$ is smaller than that of $\widehat a_{n}$.

Nevertheless, the situation may be quite different when the errors are possibly heteroscedastic. To illustrate this fact, let us assume that $f_{0}(x)=(\sqrt{x}|\log (x/e)|)^{-1}$ for $x\in (0,1]$ and  $\xi_{i,n}$ is a Gaussian random variable with variance $\sigma_{i,n}^{2}=0$ when $i\in\{2,\ldots,n\}$ while $\sigma_{1,n}^{2}=n\log^{2}(en)$.  This means that if we exclude $Y_{1,n}$, the signal is observed with no noise!  The LSE $\widetilde a_{n}$ is then distributed as a Gaussian random variable centred at $a\et$ with variance
\begin{align*}
\frac{\sigma_{1,n}^{2}f_{0}^{2}(x_{1,n})}{\cro{\sum_{i=1}^{n}f_{0}^{2}(x_{i,n})}^{2}}=\pa{\frac{1}{n}\sum_{i=1}^{n}f_{0}^{2}(x_{i,n})}^{-2}\CV{n\to +\infty}\pa{\int_{0}^{1}\frac{dx}{x\log^{2}(x/e)}}^{-2}=1.
\end{align*}
It  is therefore inconsistent while the variance of our estimator $\widehat a_{n}$, which is also Gaussian with mean $a\et$, is
\begin{align*}
\frac{\sigma_{1,n}^{2}}{n^{2}}\pa{\frac{1}{n}\sum_{i=1}^{n}f_{0}(x_{i,n})}^{-2}\quad \text{with}\quad \frac{1}{n}\sum_{i=1}^{n}f_{0}(x_{i,n})\CV{n\to +\infty}\int_{0}^{1}\frac{dx}{|\log(x/e)|\sqrt{x}}\approx 0.92.
\end{align*}
The estimator $\widetilde a_{n}$ converges thus at the rate $\sigma_{1,n}/n=(\log n)/\sqrt{n}$. Denoting by $c_{0}$ the first absolute moment of a standard Gaussian variable, we deduce that the risks of the estimators $\widetilde f=\widetilde a f_{0}$ and $\widehat f=\widehat a f_{0}$ are respectively given by
\begin{equation}\label{eq-LSvO01}
\E\cro{\ell\pa{a\et f_{0}, \widetilde f}}=\frac{c_{0}\sum_{i=1}^{n}f_{0}(x_{i,n})}{\sum_{i=1}^{n}f_{0}^{2}(x_{i,n})}\CV{n\to +\infty}c_{0}\norm{f_{0}}_{1}\quad \text{and}\quad \E\cro{\ell\pa{a\et f_{0}, \widehat f}}= \frac{c_{0}\log n}{\sqrt{n}}.
\end{equation}
Note that our estimator would remain consistent as long as  $\sigma_{1,n}^{2}=o(n^{2})$. This example illustrates the fact that the LSE may perform poorly when the errors are heteroscedastic, even in a Gaussian setting. Its risk may crucially depend on the largest value of the variances while the risk of our estimator only depends on their average which may be much smaller. This property is actually not specific to this example and will be established in greater generality in our main result.

Let us now prove~\eref{eq-OurEst}.
\begin{proof}
Let $\cF$ be the family of functions $\{af_{0},\; a\in\R\}$. Note that for every $a,b,c\in\R$, $\sgn(ac,bc)=\sgn(c)\sgn(a,b)$, hence for every $x\in (0,1]$
\[
\sgn(af_{0},bf_{0},x)=\sgn(af_{0}(x),bf_{0}(x))=\sgn(f_{0}(x))\sgn(a,b).
\]
We deduce that
\begin{align*}
T(\bsZ,af_{0},bf_{0})=&\sum_{i=1}^{n}\pa{af_{0}(x_{i,n})-Y_{i,n}}\sgn(af_{0},bf_{0},x_{i,n})\\
&=\sgn(a,b)\sum_{i=1}^{n}\pa{a|f_{0}(x_{i,n})|-Y_{i,n}\sgn(f_{0}(x_{i,n})}
\end{align*}
hence
\begin{align*}
\sup_{b\in\R}T(\bsZ,af_{0},bf_{0})&=\ab{\sum_{i=1}^{n}\pa{a|f_{0}(x_{i,n})|-Y_{i,n}\sgn(f_{0}(x_{i,n})}}\\
&=\ab{a\sum_{i=1}^{n}|f_{0}(x_{i,n})|-\sum_{i=1}^{n}Y_{i,n}\sgn(f_{0}(x_{i,n}))}.
\end{align*}
The minimum is achieved for $a$ given by \eref{eq-OurEst}.
\end{proof}

\section{Extremal elements and main results}\label{sect-CEMR}
\subsection{Extremal elements of $\cF$}
The elements $\overline f$ of $\cF$ which satisfy the following property play a central role in
our main results.
\begin{df}\label{def-extremal}
Let $\cG$ be a class of functions on a set $\cY$. We say that $\overline g\in\cG$ is extremal in
$\cG$ with degree not larger than $D\ge 1$ if the VC-dimensions of the classes of subsets of $\cY$
\[
\cA_{>}(\overline{g}) = \left\{ \{ g >\overline{g}\},\;  g\in \cG\right\}\quad \text{and}\quad \cA_{<}(\overline{g}) = \left\{ \{ g< \overline{g}\},\;  g \in \cG\right\}
\]
are not larger than $D$. The degree of an extremal element $\overline g\in\cG$ is the smallest
integer $D$ for which the above property is satisfied.
\end{df}
We recall that the VC-dimension of a class $\cA$ of subsets of $\cY$ is not larger than $D$ if
$\ab{\{B\cap A,\; A\in\cA\}}<2^{|B|}$ for every finite subset $B$ of $\cY$ with cardinality larger than $D$. In other words, there exists at least one subset of $B$ which is not of the form
$B\cap A$ with $A\in\cA$. The VC-dimension of $\cA$ is the smallest integer $D$ for which
this property is satisfied.

Going back to our regression framework with the class of functions $\cF$ on $\cX$, we denote by $\pi_{n}f$ the restriction of a function $f\in\cF$ to the subset $\cX_{n}=\{x_{1,n},\ldots,x_{n,n}\}\subset \cX$. The set $\cF_{n}=\{\pi_{n}f,\; f\in\cF\}$ is therefore a class of functions on the finite set $\cX_{n}$. For an integer $D\ge 1$, we define $\cO(D)$ as the subset of $\cF$ which gathers these  elements $\overline f\in\cF$ whose restrictions $\pi_{n}\overline f$ are extremal with degree not larger than $D$ in $\cF_{n}$. Let us now state some useful properties that we shall repeatedly use throughout this paper.
\begin{listi}
\item Since $|\cX_{n}|\le n$ (possibly smaller in case of ex-aequos), the
VC-dimension of a class of subsets of $\cX_{n}$ cannot be larger than $n$. As a consequence,
every element of $\cF_{n}$ is extremal with a degree which is not larger than $n$, 
\begin{equation}\label{eq-biais1}
\cF=\cO(n)\quad \text{and}\quad \ell(f,\cO(n))=0\quad \text{for every $f\in\cF$.}
\end{equation}
This also means that we may  restrict the values of $D$ to the set $\{1,\ldots,n\}$. 

\item It is clear that for $\overline{f}\in\cF$, 
\[
\ab{\left\{ \{f>\pi_{n}\overline{f}\},\;f \in\cF_{n}\right\}}\le\ab{ \left\{ \{f>\overline{f}\},\;f \in\cF\right\}}.
\]
This implies that the degree of $\pi_{n}\overline f$ in $\cF_{n}$ cannot be larger than that of $\overline{f}$ in $\cF$. The set $\cO(D)$ contains thus these functions which are extremal in $\cF$ with degree not larger than $D$. 
\item If  $\overline f$ is an element of $\cF$ for which $|\cA_{>}(\overline{f})|\vee |\cA_{<}(\overline{f})|\le 2^{D}$, then $\overline f$ is extremal in $\cF$ with degree not larger than $D$, hence belongs to $\cO(D)$. 
\end{listi}
In Example~\ref{exa1}, the sets $\{f>\overline f\}$ and $\{f<\overline f\}$ are unions of intervals of $\cI$. Since the  cardinality of $\cI$ is $D$, there exist no more than $2^{D}$ of such sets.  Every element $\overline f\in\cF$ is therefore extremal in $\cF$ with degree not larger than $D$.

More generally, the following result holds.
\begin{prop}\label{prop-lin}
If $\cF$ is a linear space on $\cX$ with finite dimension $d\ge 1$, all the elements of $\cF$ are extremal with degree not larger than $d+1$. In particular, $\cO(d+1)=\cF$. 
\end{prop}
\begin{proof}
This proposition is the consequence of the fact that $\cF$ is VC-subgraph with dimension at most $d+1$ (see van der Vaart and Wellner~\citeyearpar{MR1385671} Lemma~2.6.15). This property implies that the class of sets $\{\{f>t\},\; f\in\cF, t\in\R\}$ are VC with dimension not larger than $d+1$. Hence, so are the subsets $\cA_{<}(\overline f)$ and $\cA_{>}(\overline f)$ for every $\overline f\in\cF$.
\end{proof}

\subsection{The main result}
Our estimator possesses the following property.
\begin{thm}\label{thm-main1}
Assume that the random variables $Y_{1,n},\ldots,Y_{n,n}$ are independent and let $\widehat f$ be an element of $\cE_{\frc}(\bsZ,\cF)$ with $\frc\ge 0$. Then, for every $\overline f\in\cO(D)$ with $D\in\{1,\ldots,n\}$,
\begin{equation}\label{eq-main00}
\E\cro{\ell\pa{\widehat f,\overline f}}\le 2\ell\pa{f\et,\overline f}+ \kappa\inf_{p\in [1,2]}\cro{\sigma_{p}(\gx)\pa{\frac{D}{n}}^{1-\frac{1}{p}}}+\frac{\frc}{n}
\end{equation}
where $\kappa$ denotes a positive numerical constant. In particular,
\begin{align}
\E\cro{\ell\pa{f\et,\widehat f}}\le  \inf_{D\in \{1,\ldots,n\}}\cro{3\ell\pa{f\et,\cO(D)}+ \kappa\inf_{p\in [1,2]}\ac{\sigma_{p}(\gx)\pa{\frac{D}{n}}^{1-\frac{1}{p}}}}+\frac{\frc}{n}\label{eq-main01}
\end{align}
and when $Y_{1,n},\ldots Y_{n,n}$ are square integrable,
\begin{align}
\E\cro{\ell\pa{f\et,\widehat f}}\le  \inf_{D\in\{1,\ldots,n\}}\cro{3\ell\pa{f\et,\cO(D)}+ \kappa\sqrt{\pa{\frac{1}{n}\sum_{i=1}^{n}\Var(Y_{i,n})}\frac{D}{n}}}+\frac{\frc}{n}.\label{eq-main020}
\end{align}
\end{thm}

A straightforward consequence of Theorem~\ref{thm-main1} and  Proposition~\ref{prop-lin} is the following one.
\begin{cor}\label{cor-lin}
The estimator $\widehat f$ built on a linear space $\cF$ with dimension $d\ge 1$ satisfies
\begin{align}
\E\cro{\ell\pa{f\et,\widehat f}}\le 3\ell\pa{f\et,\cF}+ \kappa\inf_{p\in [1,2]}\ac{\sigma_{p}(\gx)\pa{\frac{(d+1)\wedge n}{n}}^{1-\frac{1}{p}}}+\frac{\frc}{n}.\label{eq-main01lin}
\end{align}
\end{cor}
In particular, when $\cF$ is a linear space of functions on $\cX$ with dimension $d<n$ and the errors are i.i.d.\ and square integrable with variance $\tau^{2}$, we deduce from~\eref{eq-main01lin} that (for $\frc=0)$
\begin{align}\label{eq-main01lin00}
\E\cro{\ell\pa{f\et,\widehat f}}\le 3\ell\pa{f\et,\cF}+ \kappa\tau\sqrt{\frac{d+1}{n}}
\end{align}
while the LSE satisfies 
\begin{align}\label{eq-main01lin01}
\E\cro{\ell_{2}^{2}\pa{f\et,\widehat f}}\le \ell_{2}^{2}\pa{f\et,\cF}+ \tau^{2}\frac{d}{n}.
\end{align}
Inequality \eref{eq-main01lin00} can be viewed as an analogue of \eref{eq-main01lin01} for the  loss $\ell=\ell_{1}$.


Since we allow the errors $\xi_{i,n}=Y_{i,n}-\E\cro{Y_{i,n}}$ to be possibly non-i.i.d., it is convenient hereafter to (abusively) refer  to a ``homoscedastic'' framework when they have the same $p$-th absolute moment (for a given $p$ which is specified by the context) and to a ``heteroscedastic'' one otherwise, even when $p$ is not equal to $2$. Let us now turn to some comments on Theorem~\ref{thm-main1} assuming that $\frc =0$ for the sake of simplicity.

Let us first note that~\eref{eq-main01} is interesting even when $p=1$. In this case, by using \eref{eq-biais1}, we obtain that 
\begin{equation}\label{eq-jp=1}
\sup_{f\et\in\cF}\E\cro{\ell\pa{f\et,\widehat f}}\le \kappa\sigma_{1}(\gx).
\end{equation}
When $\sigma_{1}(\gx)$ is small, the risk of our estimator is uniformly small over $\cF$ even in a situation where the diameter $\sup_{f,g\in\cF}\ell(f,g)$ of $\cF$ is infinite. 

When the design points $x_{1,n},\ldots,x_{n,n}$ are distinct and the class $\cF$ is rich enough to satisfy the property
\[
\ac{\pa{f(x_{1,n}),\ldots,f(x_{n,n})},\; f\in\cF}=\R^{n},
\]
we obtain that for every regression function $f\et$ on $\cX$ there exists $f\in\cF$ such that $(f\et(x_{1,n}),\ldots,f\et(x_{n,n}))=(f(x_{1,n}),\ldots,f(x_{n,n}))$, hence $\ell(f\et,f)=0$. It follows thus from \eref{eq-biais1} that $\ell(f\et,\cO(n))=0$. This equality is therefore satisfied for every function $f\et$ on $\cX$ and not only for those which belong to $\cF$. The supremum in \eref{eq-jp=1} may then be taken over all functions on $\cX$.

When $\sigma_{p}(\gx)<+\infty$ for some $p\in (1,2]$ and $f\et$ belongs to $\cO(D)$ for a given value of $D$, we derive from \eref{eq-main01} that
the estimator $\widehat f$ satisfies
\begin{equation}\label{eq-main010}
\E\cro{\ell\pa{f\et,\widehat f}}\le \kappa\sigma_{p}(\gx)\pa{\frac{D}{n}}^{1-\frac{1}{p}}.
\end{equation}
The quantity $\sigma_{p}(\gx)=\pa{n^{-1}\sum_{i=1}^{n}\E\cro{\ab{\xi_{i,n}}^{p}}}^{1/p}$ may depend on $n$ and can be small when the $p$-th absolute moments of the errors are small in average. When the errors are i.i.d.\ and $\E\cro{\ab{\xi_{1,n}}^{p}}=\tau^{p}<+\infty$, the right-hand side of \eref{eq-main010} is of order $\tau(D/n)^{1-1/p}$. For $p\in (1,2)$, the rate $n^{-1+1/p}$ cannot be improved in general because of the following result. 
\begin{prop}\label{prop-vitesse}
Let $p\in (1,2)$ and $\etc{Y}$ i.i.d.\ random variables with mean $m$.  
\begin{listi}
\item If $\E\cro{|Y_{1}|^{p}}<+\infty$ then 
\[
n^{1-1/p}\ab{\frac{1}{n}\sum_{i=1}^{n}Y_{i,n}-m}\CV{n\to +\infty} 0\quad a.s.
\]
\item If there exists $c>0$ such that 
\[
\E\cro{\ab{\frac{1}{n}\sum_{i=1}^{n}Y_{i,n}-m}}\le cn^{-(1-1/p)}\quad \text{then for every $q\in (1,p)$}\quad \E\cro{|Y_{1}|^{q}}<+\infty.
\]
\end{listi}
\end{prop}
\begin{proof}
The result in $(i)$ is a consequence of Theorem~2.5.12 in Durrett~\citeyearpar{MR3930614} while $(ii)$ follows from Markov's inequality and Theorem~1 in Baum and Katz~\citeyearpar{MR198524}.
\end{proof}
When the errors are no longer i.i.d.\ but satisfy $\E\cro{\ab{\xi_{i,n}}^{p}}\le \tau^{p}<+\infty$ except may be for a small enough number of them which may satisfy $\E\cro{\ab{\xi_{i,n}}^{p}}\le  n\tau^{p}$, the quantity $\sigma_{p}(\gx)$ remains of order $\tau$ and the risk of our estimator of the same order as in the i.i.d.\ case. This property shows the stability of our estimator to a possible ``heteroscedasticity'' of the errors.

Other examples of heteroscedastic frameworks have already been presented in our Examples~\ref{exai3} and~\ref{exai4}.

A interesting consequence of \eref{eq-main01} lies in the fact that  when $f\et$ does not belong to $\cO(D)$  but is only close enough to an element of this set,  the risk bound in \eref{eq-main01} remains of the same order as that discussed in \eref{eq-main010}. In addition, the infimum over $D$ in \eref{eq-main01} indicates that this risk bound provides the best trade-off  between the approximation term $\ell(f\et,\cO(D))$, which is nonincreasing with respect to $D$, and the quantity $\sigma_{p}(\gx)(D/n)^{1-1/p}$, which is increasing with $D$ when $p>1$. It shows how the estimator adapts to some specific features of the unknown regression function. 

Another interesting consequence of \eref{eq-main01} lies in the infimum with respect to $p\in [1,2]$. Since $\sigma_{p}(\gx)$ and $( D/n)^{1-1/p}$ vary in opposite directions  with $p$, the risk bound automatically provides the best compromise between these two factors. This adaptation property with respect to the amount of integrability of the errors is to our knowledge new. It is illustrated in Example~\ref{exa-6} below.

\begin{exa}[Example~\ref{exa2} continued]\label{exa-6}
In this statistical framework, the $p$-th moment of $|\xi_{1}|$ exists for every $p\in [1,2)$ but  it grows to infinity as $p$ approaches 2. We obtain that
\begin{equation}\label{eq-exa6}
\inf_{p\in [1,2]}\cro{\sigma_{p}(\gx)\pa{\frac{D}{n}}^{1-\frac{1}{p}}}=\sqrt{\frac{D(eL_{n})^{\beta}}{n}}\quad \text{for $n\ge D(e/2)^{\beta}$}
\end{equation}
where $L_{n}$ is the unique solution of the equation $\beta^{-1}\log(n/D)=L-\log L-1$ for $L\ge 2$. When $n/D\ge \exp[\beta(e-2)]$ then
\[
\frac{1}{\beta}\log\pa{\frac{n}{D}}+1\le L_{n}\le \frac{e}{e-1}\cro{\frac{1}{\beta}\log\pa{\frac{n}{D}}+1}.
\]
For large enough values of $n/D$, $L_{n}$ is of order $\log(n/D)$, hence the left-hand side of \eref{eq-exa6} of order $\sqrt{D\log(n/D)/n}$. Up to a logarithmic factor, the risk of the estimator is of the same order as that we would get for errors which are square integrable.

Let us now prove \eref{eq-exa6}. By doing the change of variables $t=(2-p)\log u$, $dt=(2-p)u^{-1}du$, we obtain that for $p\in [1,2)$,
\begin{align*}
\E\cro{|\xi_{1}|^{p}}&=\frac{2^{\beta}}{G(\beta)}\int_{1}^{+\infty}\frac{\log^{\beta-1} u}{u^{3-p}}du=\frac{2^{\beta}}{(2-p)^{\beta}G(\beta)}\int_{0}^{+\infty}t^{\beta-1}e^{-t}dt=\pa{\frac{2}{2-p}}^{\beta}.
\end{align*}
Consequently, $\sigma_{p}(\gx)=[2/(2-p)]^{\beta/p}$. Setting $z=(n/D)^{1/\beta}$ and $L=2/(2-p)\in [2,+\infty)$, hence $p=2(1-1/L)$, we obtain that
\begin{align*}
\sigma_{p}(\gx)\pa{\frac{D}{n}}^{1-1/p}&=z^{-\beta}\pa{\frac{2z}{2-p}}^{\beta/p}=z^{-\beta}\exp\cro{\beta f(L)}\quad \text{with}\quad f(L)=\frac{L\log(Lz)}{2(L-1)}.
\end{align*}
Since $f'(L)=[2(L-1)]^{-1}\pa{1-\log(Lz)/(L-1)}$, $f$ reaches its minimum at the solution $L_{n}$ of the equation $\log z=L-\log L-1$ on $[2,+\infty)$, which exists and is unique when $z\ge (e/2)$, i.e.\ $n\ge D(e/2)^{\beta}$. Since $\log(L_{n}z)=L_{n}-1$ and $L_{n}=\log(eL_{n}z)$ we get
\begin{align*}
\inf_{p\in [1,2)}\cro{\sigma_{p}(\gx)\pa{\frac{D}{n}}^{1-1/p}}&=z^{-\beta}\inf_{L\ge 2}\exp\cro{\beta f(L)}=z^{-\beta}\exp\cro{\beta\frac{L_{n}\log(L_{n}z)}{2(L_{n}-1)}}\\
&=z^{-\beta/2}(eL_{n})^{\beta/2}
\end{align*}
which is \eref{eq-exa6}.
\end{exa}

\section{From approximation theory to risk bounds}\label{sect-FATH}
Throughout this section, we assume that $f\et$ belongs to $\cF$. If it were not the case, the risk bound that we establish here would not inflate by more than the additional term $3\ell(f\et,\cF)$. For the sake of simplicity, we also assume that $\sigma_{p}(\gx)$ can be bounded independently of $\gx$ and $n$ by some positive number $\sigma_{p}$ when $\sigma_{p}(\gx)<+\infty$ and we use the convention $\sigma_{p}=+\infty$ otherwise. It follows from Theorem~\ref{thm-main1} that the risk of the estimator $\widehat f$ built on $\cF$  satisfies $\E\cro{\ell\pa{f\et,\widehat f}}\le B_{n}(f\et)+\frc/n$ where
\begin{equation}\label{eq-riskfchap1}
B_{n}(f\et,\cF)=\inf_{D\ge 1}\ac{3\ell\pa{f\et,\cO(D)}+\kappa\inf_{p\in[1,2]}\cro{\sigma_{p}\pa{\frac{D}{n}}^{1-\frac{1}{p}}}}.
\end{equation}
For some classes $\cF$ of interest, our aim is to specify the quantity $B_{n}(f\et,\cF)$ by describing the subsets $\{\cO(D), D\ge 1\}$ of extremal functions as well as their approximation properties with respect to the elements of $\cF$. Let us now turn to our examples.

\subsection{$k$-piecewise monotone functions}
Hereafter $\cX$ denotes a nontrivial interval of $\R$ and we assume that the points $x_{i,n}$ are labelled in increasing order, that is, $x_{1,n}\le \ldots\le x_{n,n}$. Given a positive integer $k$, we consider the class $\cF=\cM_{k}$ of functions which are called {\em $k$-piecewise monotone} on $\cX$ which means that they are monotone on each element of a partition $\cJ$ of $\cX$ into at most $k$ nontrivial intervals. For $f\in\cM_{k}$, {\em the $\cJ$-variation} of $f$ on $\cX$ is the quantity
\begin{equation}\label{def-variation}
V_{\cJ}(\gx,f)=\pa{\sum_{J\in\cJ}\sqrt{\frac{n_{J}V_{J}(f)}{n}}}^{2}\quad \text{where}\quad V_{J}(f)= \sup_{x\in J}f(x)-\mathop{\vphantom{p}\inf}_{x\in J}f(x)
\end{equation}
and $n_{J}$ is the number of points $x_{i,n}$ that fall in $J\in\cJ$. It follows from Cauchy-Schwarz's inequality that $V_{\cJ}(\gx,f)\le V(f)$ where
\begin{equation}\label{eq-tvar}
V(f)=\sum_{J\in\cJ}V_{J}(f)
\end{equation}
is the {\em (total) variation} of $f$ on $\cX$. This latter quantity is independent of the design points $x_{i,n}$ while the $\cJ$-variation does depend on them. Note that $V_{\cJ}(\gx,f)=0$ if and only if $f$ is constant on each element $J$ of the partition $\cJ$ which contains at least a point $x_{i,n}$.

With this notation, $\cM_{1}$ is the set of functions  which are monotone on $\cX$. The class $\cM_{2}$ contains those which are unimodal. For general values of $k\ge 1$, $\cM_{k}$ contains the functions which are {\em $k$-piecewise constant} on $\cX$, that is, piecewise constant on a partition of $\cX$ into at most $k$ nontrivial intervals. Such functions are sometimes called {\em regressograms} with $k$ bins (or a  spline with degree 0 on $k+1$ knots). The following proposition is a consequence of Theorem~\ref{thm-VC} applied to $\cF=\cM_{k}$ and $r=1$.

\begin{prop}\label{prop-extrMonok}
Let $\overline f$ be a function which is both $k$-piecewise monotone and $K$-piecewise constant on $\cX$ with $K\ge 1$. Then $\overline f$ is extremal in $\cM_{k}$ with degree not larger than $2k+K-1$.
\end{prop}

The set $\cO(D)$ with $D=2k+K-1$ contains thus the functions which are both $k$-piecewise monotone and $K$-piecewise constant. Denoting by $\cO_{\text{pc}}(K)\subset \cM_{k}$ the set of such functions, we immediately derive from Proposition~\ref{prop-extrMonok} and \eref{eq-riskfchap1} that for every regression function $f\et$ on $\cX$,
\begin{equation}\label{eq-adaptmono}
B_{n}(f\et,\cM_{k})\le \inf_{K\ge 1}\ac{3\ell\pa{f\et,\cO_{\text{pc}}(K)}+\kappa\inf_{p\in[1,2]}\cro{\sigma_{p}\pa{\frac{2k+K-1}{n}}^{1-\frac{1}{p}}}}.
\end{equation}
In particular,
\begin{equation}\label{eq-cor-monok2}
B_{n}(f\et,\cM_{k})\le \kappa\inf_{p\in[1,2]}\cro{\sigma_{p}\pa{\frac{2k+K-1}{n}}^{1-\frac{1}{p}}}\quad \text{for every $f\et\in \cO_{\text{pc}}(K)$.}
\end{equation}
If $f\et$ is a regressogram with at most $k$ bins and the errors $Y_{i,n}-\E\cro{Y_{i,n}}$ admit a $p$-th moment with $p\in (1,2]$, $f\et$ belongs to  $\cO_{\text{pc}}(k)$ and our estimator $\widehat f$ on $\cM_{k}$ converges thus to $f\et$ at a rate which is at least $n^{-(1-1/p)}$.

The following result describes the approximation properties of the elements of $\cO_{\text{pc}}(K)$ with respect to monotone functions on $\cX$ for the loss $\ell$. Its proof is postponed to Section~\ref{sect-pf-prop-approxMono}.
\begin{prop}\label{prop-approxMono}
Let $f\in \cM_{k}$ with $k\ge 1$. For every $\gamma>0$, there exists an element $\overline f$ in $\cO_{\text{pc}}(\PEI{k+\gamma})$ such that
\begin{equation}\label{eq-approxcm}
\frac{1}{n}\sum_{i=1}^{n}\ab{f(x_{i,n})-\overline f(x_{i,n})}\le \frac{V_{\cJ}(\gx,f)}{\gamma}.
\end{equation}
\end{prop}

We deduce the following result.
\begin{cor}\label{cor-monok}
Let $k\ge 1$. Assume that $\sigma_{p}<+\infty$ for some $p\in [1,2]$ and  $f\et$ is an element of $\cM_{k}$ with a finite $\cJ$-variation $V_{\cJ}=V_{\cJ}(\gx,f)$ defined by~\eref{def-variation}. Then, for some numerical constant $C>0$,
\begin{equation}\label{eq-cor-monok}
B_{n}(f\et,\cM_{k})\le C\cro{\sigma_{p}^{\frac{1}{s+1}}\pa{\frac{V_{\cJ}}{n}}^{\frac{s}{s+1}}+\sigma_{p}\pa{\frac{3k-1}{n}}^{s}}\quad \text{with $s=1-\frac{1}{p}\in [0,1/2]$.}
\end{equation}
\end{cor}

In particular, when the variation $V(f\et)$ of $f\et$ is finite on $\cX$, it follows from the inequality $V_{\cJ}(\gx,f)\le V(f\et)$ that $\widehat f$ converges to $f\et$ at a rate which is at least  $n^{-(p-1)/(2p-1)}$ when $\sigma_{p}<+\infty$.

From a nonasymptotic point of view, it is interesting to notice that when the $\cJ$-variation $V_{\cJ}(\gx,f\et)$ of $f\et$ is small enough, so that the first term in the right-hand side of \eref{eq-cor-monok} becomes sufficiently small compared to the second one, the risk bound that we get is of the same order as that we would get for estimating a regressogram on $k$ bins (see \eref{eq-cor-monok2} with $K=k$).
\begin{proof}[Proof of Corollary~\ref{cor-monok}]
Let $p\in [1,2]$ such that $\sigma_{p}<+\infty$.  It follows from Proposition~\ref{prop-approxMono} and \eref{eq-adaptmono} that for every $\gamma>0$
\begin{align*}
B_{n}(f\et,\cM_{k})&\le \frac{3V_{\cJ}}{\gamma}+\kappa\sigma_{p}\pa{\frac{2k+\PEI{k+\gamma}-1}{n}}^{s}.
\end{align*}
When $V_{\cJ}=0$, we choose $\gamma=1/2$ so that $\PEI{k+\gamma}=k$ and we get that $B_{n}(f\et)\le \kappa\sigma_{p}[(3k-1)/n]^{1-1/p}$. Otherwise, by using the fact that $u\mapsto u^{s}$ is subadditive and the inequality $\PEI{k+\gamma}\le k+\gamma$,  we obtain that
\begin{align*}
B_{n}(f\et,\cM_{k})&\le \frac{3V_{\cJ}}{\gamma}+\kappa\sigma_{p}\cro{\pa{\frac{\gamma}{n}}^{s}
+\pa{\frac{3k-1}{n}}^{s}}.
\end{align*}
The result follows by choosing $\gamma=(V_{\cJ}/(\kappa \sigma_{p}))^{1/(s+1)}n^{s/(s+1)}$.
\end{proof}

%
%
%
%
%
%
%

\subsection{$k$-piecewise monotone-convex-concave functions}
In this section, $\cX$ denotes a nontrivial compact interval of $\R$. Hereafter, we say that a function $f$ is {\em convex-concave} on $[a,b]$, with $a<b$,  if it is either convex or concave on this interval. Given a continuous convex-concave function $f$ on an interval $[a,b]$, we may denote by $f_{r}'(x)$ and $f_{l}'(x)$ the right and left derivatives of $f$ at $x\in [a,b]$ respectively. We allow these quantities to take the values $\pm \infty$. For $f$ being additionally monotone, we define the {\em linear index} of $f$ as the quantity
\[
\Gamma(f)=1-\frac{1}{2}\cro{\frac{\ab{f'_{r}(a)}\wedge \ab{f'_{l}(b)}}{|\Delta|}+\frac{|\Delta|}{\ab{f'_{r}(a)}\vee\ab{f'_{l}(b)}}}\quad \text{with}\quad \Delta=\frac{f(b)-f(a)}{b-a}.
\]
In this formula, we use the convention $0/0=1$ and $1/(+\infty)=0$. The linear index $\Gamma(f)$ belongs to $[0,1]$ and equals 0 if and only if $f$ is affine on $[a,b]$. In contrast, the linear index of the function $=(1-10^{10}x)\1_{[0,10^{-10}]}$, which is nonincreasing and ``strongly'' convex on $[0,1]$, is $\Gamma(f)=1-10^{-10}/2$ and is therefore  close to 1. The linear index measures how far the graph of a continuous, monotone convex-concave function $f$ on an interval $[a,b]$ differs from the chord joining $(a,f(a))$ to $(b,f(b))$.

Given $k\ge 1$, we say that $f$  is {\em $k$-piecewise monotone-convex-concave} on $\cX$ if it is {\em continuous} on $\cX$ and if there exists a partition of $\cX$ into at most $k$ nontrivial intervals on the closure of each of which $f$ is monotone and convex-concave. We denote by $\cC_{k}$ the class of such functions. The set $\cC_{2}$ contains the functions which are continuous and either convex or concave on $\cX$. Those with an $S$-shape (first convex and then concave on $\cX$) belong to $\cC_{4}$. For $k\ge 1$, $\cC_{k}$ contains all the functions which are continuous and piecewise linear on a partition of $\cX$ into at most $k$ intervals. We shall call these functions {\em $k$-piecewise linear on $\cX$}. They are splines of degree at most 1 with $k+1$ knots. 
We start with the following result which derives from Theorem~\ref{thm-VC} applied to $\cF=\cC_{k}$ and $r=2$.
\begin{prop}\label{prop-extrcck}
A function $\overline f$ of $\cC_{k}$ which is $K$-piecewise linear on $\cX$ with $K\ge 1$ is extremal in $\cC_{k}$ with degree not larger than $3k+2(K-1)$.
\end{prop}

The set $\cO(D)$ with $D=3k+2(K-1)$ contains thus these $K$-piecewise linear functions which belong to $\cC_{k}$. Let us denote by $\cO_{\text{pl}}(K)$ the set of such functions. We immediately deduce from Proposition~\ref{prop-extrcck} and \eref{eq-riskfchap1} that for every regression function $f\et$ on $\cX$,
\begin{equation}\label{eq-adaptcon}
B_{n}(f\et,\cC_{k})\le \inf_{K\ge 1}\ac{3\ell\pa{f\et,\cO_{\text{pl}}(K)}+\kappa\inf_{p\in[1,2]}\cro{\sigma_{p}\pa{\frac{3k+2(K-1)}{n}}^{1-\frac{1}{p}}}}.
\end{equation}
In particular, if $f\et$ belongs to $\cO_{\text{pl}}(K)$ and $\sigma_{p}<+\infty$ for some $p\in [1,2]$,
\begin{equation}\label{eq-pcK}
B_{n}(f\et,\cC_{k})\le \kappa\sigma_{p}\pa{\frac{3k+2(K-1)}{n}}^{1-\frac{1}{p}}.
\end{equation}
This inequality applies to every function $f\et$ which is $k$-piecewise linear on $\cX$ as it automatically belongs to $\cC_{k}$. For such an extremal function, the estimator $\widehat f$ converges at a rate which is at least $n^{1/p-1}$ when $p\in (1,2]$.

Given a function $f\in \cC_{k}$ associated with a partition $\cJ$ of the interval $\cX$ , we define the $\cJ$-{\em variation-index} of $f$ as
\begin{equation}\label{def-W}
W_{\cJ}(f)=\cro{\sum_{J\in\cJ}\pa{\frac{n_{J}V_{J}(f)\Gamma_{J}(f)}{n}}^{\frac{1}{3}}}^{3}
\end{equation}
where $n_{J}$ is the number of design points that fall in $J$, $V_{J}(f)$ the variation of $f$ on $J$ and $\Gamma_{J}(f)$ the linear index of $f$ on the closure of $J$. Note that $W_{\cJ}(f)=0$ if and only if $f$ is affine on every interval $J$ that contains at least one design point $x_{i,n}$. Besides, it follows from H\"older's inequality that the $\cJ$-variation-index of $f$ satisfies
%
\begin{equation}\label{ine-W}
W_{\cJ}(f)\le \pa{\sum_{J\in\cJ}\frac{n_{J}}{n}}\pa{\sum_{J\in\cJ}V_{J}(f)}\pa{\sum_{J\in\cJ}\Gamma_{J}(f)}=V(f)\Gamma_{\cJ}(f)
\end{equation}
where $V(f)=\sum_{J\in\cJ}V_{J}(f)$ is the variation of $f$ on $\cX$ and $\Gamma_{\cJ}(f)=\sum_{J\in\cJ}\Gamma_{J}\le |\cJ|\le k$.

The following approximation result is proven in Section~\ref{pf-prop-mcc}.
\begin{prop}\label{prop-approx-mcc}
Let $\cX$ be a compact interval of length $L>0$, $x_{1,n},\ldots,x_{n,n}$ equispaced points  in $\cX$ such that $x_{i+1}-x_{i,n}=L/n$ for $i\in\{1,\ldots,n-1\}$, $f$ an element of $\cC_{k}$ with  variation $V=V(f)$ defined by~\eref{eq-tvar} and $\cJ$-variation-index $W=W_{\cJ}(f)$ defined by~\eref{def-W}. Then, for every $\gamma>0$, there exists an element $\overline f=\overline f(\gamma)$ in $\cO_{\text{pl}}(\PEI{2(k+\gamma)})$  such that
\begin{equation}\label{eq-approx-mcc}
\frac{1}{n}\sum_{i=1}^{n}\ab{f(x_{i,n})-\overline f(x_{i,n})}\le \frac{W}{\gamma^{2}}+\frac{2V}{n}.
\end{equation}
\end{prop}

We deduce
\begin{cor}\label{cor-approx-mcc}
Let $\cX$ be some compact interval of length $L>0$, $k$ a positive integer and $x_{1,n},\ldots,x_{n,n}$ equispaced points  in $\cX$, that is, $x_{i+1}-x_{i,n}=L/n$ for $i\in\{1,\ldots,n-1\}$. Assume that $\sigma_{p}<+\infty$ for some $p\in [1,2]$ and that $f\et$ belongs to $\cC_{k}$ with  variation $V=V(f\et)$ defined by~\eref{eq-tvar} and $W=W_{\cJ}(f\et)$ defined in~\eref{def-W}. Then, there exists a numerical constant $C>0$ such that
\[
B_{n}(f\et,\cC_{k})\le C\cro{\sigma_{p}^{\frac{2}{s+2}}\pa{\frac{\sqrt{W}}{n}}^{\frac{2s}{s+2}}+\sigma_{p}\pa{\frac{7k-2}{n}}^{s}+\frac{V}{n}}\quad \text{with $s=1-\frac{1}{p}$.}
\]
\end{cor}
When $W$ and $p$ are positive, the rate of convergence of $\widehat f$ to $f\et$ is therefore at least $n^{-2s/(s+2)}$. From a nonasymptotic point of view, when $W$ is small enough the risk bound is of the same order
as that we would have for estimating a $k$-piecewise linear function. The quantity $W$ is small  when either the variation or the linear index of the regression function $f\et$ is small on every interval $J$ that contains a design point.

To illustrate Corollary~\ref{cor-approx-mcc}, let us assume that  $\cX=[0,1]$ , $x_{i,n}=(i-1/2)/n$ for $i\in\{1,\ldots,n\}$, and $f\et$ is continuous and convex-concave on $\cX$ with variation $V=V(f\et)<+\infty$. Applying Corollary~\ref{cor-approx-mcc} with $k=2$ and using the fact that $W\le 2V$ by \eref{ine-W}, we obtain that
\[
B_{n}(f\et,\cC_{k})\le C\cro{\sigma_{p}^{\frac{2}{s+2}}\pa{\frac{\sqrt{2V}}{n}}^{\frac{2s}{s+2}}+\sigma_{p}\pa{\frac{12}{n}}^{s}+\frac{V}{n}}.
\]
%

\begin{proof}[Proof of Corollary~\ref{cor-approx-mcc}]
It follows from Proposition~\ref{prop-approx-mcc} and \eref{eq-adaptcon} that for any $\gamma>0$,
\begin{align*}
B_{n}(f\et,\cC_{k})&\le \frac{3W}{\gamma^{2}}+\kappa\sigma_{p}\pa{\frac{3k+2(\PEI{2k+2\gamma}-1)}{n}}^{s}+\frac{6V}{n}.
\end{align*}
If $W=0$, we choose $\gamma=1/3$ and get that $B_{n}(f\et,\cC_{k})\le \kappa\sigma_{p}((7k-2)/n)^{s}+6V/n$. If $W>0$, the result follows by choosing $\gamma=\pa{W/(\kappa\sigma_{p})}^{\frac{1}{s+2}}n^{\frac{s}{s+2}}$ and using the fact that
\begin{align*}
B_{n}(f\et,\cC_{k})&\le \frac{3W}{\gamma^{2}}+\kappa\sigma_{p}\cro{\pa{\frac{4\gamma}{n}}^{s}+\pa{\frac{7k-2}{n}}^{s}}+\frac{6V}{n}.
\end{align*}
\end{proof}

The previous result applies to equispaced design points. What happens when the design is not equispaced? This is what we want to investigate in the remaining part of this section. For the sake of simplicity, we restrict ourselves to the set $\cC_{1}$ of monotone-convex-concave functions on the interval $\cX$. We start with the following definitions.
\begin{df}
Let $K$ be a positive integer and $f$ be a continuous, monotone and convex-concave function on the compact interval $[a,b]$, $a<b$. We say that $l$ is a $K$-linear interpolation of $f$ on $[a,b]$ if there exists a subdivision $a_{0}=a<a_{1}<\ldots<a_{K}=b$ such that $l$ is affine on $[a_{i-1},a_{i}]$ and coincides with $f$ at $a_{i-1}$ and $a_{i}$, i.e. $f(a_{i-1})=l(a_{i-1})$ and $f(a_{i})=l(a_{i})$,  for all $i\in\{1,\ldots,K\}$.
\end{df}
If $l$ is a linear interpolation of a continuous, monotone and convex-concave function $f$ on $[a,b]$, then $l$ is also continuous, monotone and convex-concave on this interval. In fact, $l$ has the same monotonicity as $f$ and if $f$ is convex (respectively concave) so is $l$.

\begin{df}\label{def_reg_measure}
Let $\lambda$ denote the Lebesgue measure on $\R$. A modulus $w$ of a Borel measure $\mu$ on $\R$ is a nonnegative, nondecreasing, concave function on $[0,+\infty)$ that satisfies
\[
\mu(I)\le w(\lambda(I))\quad \text{for every bounded interval $I\subset \R$.}
\]
\end{df}
In this definition, we do not require that $w(0)=0$ as is usually the case for a modulus {\em of continuity}. By doing so, we allow the measure $\mu$ to be atomic. This is the situation we have in mind as we shall mainly apply it to $\mu=n^{-1}\sum_{i=1}^{n}\delta_{x_{i,n}}$. 

\begin{prop}\label{propw}
Let $g$ be an increasing continuous function on a compact interval $\cX$ with (concave) modulus of continuity $\overline w$. If $x_{i,n}=g^{-1}(i/n)$ for all $i\in\{1,\ldots,n\}$ then $w:u\mapsto n^{-1}+\overline w(u)$ is a modulus for $Q=n^{-1}\sum_{i=1}^{n}\delta_{x_{i,n}}$. 
\end{prop}
\begin{proof}
The length of an interval $J$ is at least the number of integers that fall in this interval minus 1. Given a bounded interval $I\subset \R$ with endpoints $a<b$ we may apply this result to the interval $J=\{ng(x),\; x\in I\}$ the length of which is $n\pa{g(b)-g(a)}$ and obtain that 
\begin{align*}
nQ(I)-1&=\sum_{i=1}^{n}\1_{x_{i,n}\in I}-1=\sum_{i=1}^{n}\1_{(i/n)\in g(I)}-1=\sum_{i=1}^{n}\1_{i\in J}-1\\
&\le \lambda\pa{J}=n\pa{g(b)-g(a)}\le n\overline w(|a-b|).
\end{align*}
This proves the proposition.
\end{proof}
\begin{thm}\label{thm-approx-con}
Let $Q$ be a probability measure on the bounded interval $[a,b]$, $a<b$, $K$ a positive integer and $f$ a continuous, monotone, convex-concave function on $[a,b]$ with variation $V=\ab{f(b)-f(a)}$ and linear index $\Gamma\in [0,1]$. We set $\underline \delta=|f'_{r}(a)|\wedge |f'_{l}(b)|$ and $\overline \delta=|f'_{r}(a)|\vee |f'_{l}(b)|\in [0,+\infty]$.
\begin{listi}
\item If the variation $\ab{f'_{r}(a)-f'_{l}(b)}$ of the derivatives of $f$ is finite on $[a,b]$, there exists a $K$-linear interpolation $\overline f$ of $f$ such that
\begin{align}\label{eq-approx-con00}
\int_{[a,b]}\ab{f-\overline f}dQ&\le \frac{\ab{f'_{r}(a)-f'_{l}(b)}(b-a)}{4K^{2}}.
\end{align}
\item If $Q$ has a modulus $w$ that satisfies for all $x>0$
\[
\Psi(x)=\int_{0}^{x}\frac{\overline w(t)}{t}dt<+\infty\quad \text{with $\overline w(\cdot)=w(\cdot)-w(0)$,}
\]
then for every $\delta\in [\underline \delta,\overline \delta]\cap (0,+\infty)$ there exists a $(2K)$-linear interpolation $\overline f$ of $f$ such that
\begin{equation}\label{eq-approx-con01}
\int_{[a,b]}\ab{f-\overline f}dQ\le V w(0)+ \frac{\pa{\delta-\underline \delta}(b-a)}{4K^{2}}+\frac{V}{K}\cro{\Psi\pa{\frac{V}{2K\delta}}-\Psi\pa{\frac{V}{2K\overline \delta}}}.
\end{equation}
\end{listi}
In particular, if $Q$ has a modulus of the form $w(u)=w(0)+A u^{\alpha}/(b-a)^{\alpha}$ with $A\ge 1$ and $\alpha\in (0,1]$,  there exists a $(2K)$-linear interpolation $\overline f$ of $f$ such that
\begin{align}\label{eq-approx-con}
\int_{[a,b]}\ab{f-\overline f}dQ&\le V\cro{w(0)+ \pa{\frac{2^{1-\alpha}A}{\alpha K^{1+3\alpha}}}^{\frac{1}{1+\alpha}}}.
\end{align}
Furthermore, if $\alpha=1=A$ then $\overline f$ satisfies
\begin{align}\label{eq-approx-con2}
\int_{[a,b]}\ab{f-\overline f}dQ&\le V\cro{w(0)+\frac{\Gamma}{K^{2}}}.
\end{align}
\end{thm}
The proof of this result is postponed to Section~\ref{sect-pf-thmapprox}. We deduce the following corollary.
\begin{cor}\label{cor-mcc-irr}
Assume that the regression function $f\et$ is continuous, monotone and convex-concave on $\cX=[a,b]$, $a<b$, with variation $V=\ab{f\et(a)-f\et(b)}$ and that $\sigma_{p}<+\infty$ for some $p\in [1,2]$.
\begin{listi}
\item If the variation $V'=\ab{(f\et)'_{r}(a)-(f\et)'_{l}(b)}$ of the derivatives of $f\et$ is finite on $[a,b]$, there exists a numerical constant $C>0$ such that
\begin{equation}\label{eq-vp-Vf}
B_{n}(f\et,\cC_{1})\le C \cro{\sigma_{p}^{\frac{2}{s+2}}\pa{\frac{V'(b-a)}{n^{2}}}^{\frac{s}{s+2}}+\frac{\sigma_{p}}{n^{s}}}\quad \text{with $s=1-1/p$.}
\end{equation}
\item If the probability $Q=n^{-1}\sum_{i=1}^{n}\delta_{x_{i,n}}$ has a modulus $w$ of the form $w(u)=w(0)+A(u/(b-a))^{\alpha}$ for $u\ge 0$ with $A\ge 1$ and $\alpha\in (0,1]$, there exists a numerical constant $C>0$ such that
\begin{equation}\label{eq-vp}
B_{n}(f\et,\cC_{1})\le C\cro{\sigma_{p}^{\frac{\beta}{s+\beta}}\pa{\frac{c_{\alpha}(A)V}{n^{\beta}}}^{\frac{s}{s+\beta}}+Vw(0)+\frac{\sigma_{p}}{n^{s}}}\quad \text{with $\beta=\frac{1+3\alpha}{1+\alpha}\in (1,2]$}
\end{equation}
and $c_{\alpha}(A)=(2^{1-\alpha}A/\alpha)^{1/(1+\alpha)}$.
\end{listi}
\end{cor}
In most of our applications, $w(0)$ is of order $1/n$. When $p=2$ and the variation $V'$  is finite on $[a,b]$,  the estimator converges at a rate which is at least  $n^{-2/5}$, independently of the locations of the design points. When $V'=+\infty$, this is no longer the case and the rate that we get does depend on the modulus of $Q$. Provided that $w(0)$ is of order $1/n$ and $p=2$, this rate is at least  $v_{n}=n^{-\beta/(2\beta+1)}$ with $\beta=\beta(\alpha)$ given by \eref{eq-vp}. If $\alpha=1$, which means that the design is close to equispaced, the rate is at least $n^{-2/5}$. When $\alpha$ is close to 0, the design is very irregular,  $\beta$ is close to 1 and $v_{n}$ close to $n^{-1/3}$ which is the rate we would get for estimating a regression function which is only monotone on $[a,b]$.

\begin{proof}[Proof of Corollary~\ref{cor-mcc-irr}]
Let $K$ be a positive integer to be chosen later on. Under $(i)$, there exists a $K$-linear interpolation $\overline f$ of $f\et$ which satisfies \eref{eq-approx-con00}. Since $\overline f$ is an element of $\cO(2K+1)$ by Proposition~\ref{prop-extrcck}, we derive from~\eref{eq-adaptcon} that
\begin{align*}
B_{n}(f\et,\cC_{1})&\le \frac{3V'(b-a)}{4K^{2}}+\kappa\sigma_{p}\pa{\frac{2K+1}{n}}^{s}.
\end{align*}
We obtain~\eref{eq-vp-Vf} by choosing $K=\PES{Z_{1}}\ge 1$ with $Z_{1}=\pa{V'(b-a) n^{s}/\sigma_{p}}^{1/(2+s)}$.
Under $(ii)$, we use \eref{eq-approx-con} where $\overline f$ is now a $(2K)$- linear interpolation of $f\et$. Since it belongs to $\cO(4K+1)$, we obtain that
\begin{align*}
B_{n}(f\et,\cC_{1})&\le 3\pa{Vw(0)+ \frac{c_{\alpha}(A)V}{K^{\beta}}}+\kappa\sigma_{p}\pa{\frac{4K+1}{n}}^{s}.
\end{align*}
Inequality \eref{eq-vp} is obtained by taking $K=\PES{Z_{2}}\ge 1$ with $Z_{2}=\pa{c_{\alpha}(A)Vn^{s}/\sigma_{p}}^{1/(\beta+s)}$.
\end{proof}

\subsection{Monotone single index model}
In this section, we consider the class $\cF=\cS$ of functions $f$ on $\cX=\R^{\frm}$, $\frm\ge 1$, of the form $x\mapsto \phi \pa{\scal{\theta}{x}}$ where $\phi$ is monotone on $\R$ and $\theta$ an element of the unit sphere $\sS_{\frm}$ of $\R^{\frm}$.  The function $\phi$ associated with $f$ will be called {\em a link function} and we set $\phi(-\infty)=\lim_{z\to -\infty}\phi(z)$ and $\phi(+\infty)=\lim_{z\to +\infty}\phi(z)$.

\begin{prop}\label{prop-VCSingle}
If $\overline f=\overline \phi(\scal{\overline \theta}{\cdot})$ is an element of $\cS$ with a link function $\overline \phi$ which is $K$-piecewise constant on $\R$, then $\overline f$ is extremal in $\cS$ with degree not larger than $(\frm+1)K$.
\end{prop}
We denote by $\cO_{\text{s-pc}}(K)$ the elements $\overline f\in \cS$ which are of the form  described in Proposition~\ref{prop-VCSingle}. If $f\et\in\cO_{\text{s-pc}}(K)$, we deduce from \eref{eq-riskfchap1} that
\begin{equation}\label{eq-cor-monosi}
B_{n}(f\et,\cS)\le \kappa\inf_{p\in [1,2]}\cro{\sigma_{p}\pa{\frac{(\frm+1)K}{n}}^{1-\frac{1}{p}}}.
\end{equation}

For more general regression functions $f\et$,  the following result holds.
\begin{cor}\label{corr-single}
Assume that  $\sigma_{p}$ is finite for some $p\in [1,2]$ and that the regression function $f\et=\phi\et(\scal{\theta\et}{\cdot})$ belongs to $\cS$ with a monotone link function $\phi\et$ the variation $V=|\phi\et(+\infty)-\phi\et(-\infty)|$ of which is finite on $\R$. Then there exists a numerical constant $C>0$ such that
\begin{equation}\label{eq-corr-single}
B_{n}(f\et,\cS)\le C\cro{\sigma_{p}^{\frac{1}{s+1}}\pa{\frac{V(\frm+1)}{n}}^{\frac{s}{s+1}}+\sigma_{p}\pa{\frac{\frm+1}{n}}^{s}}\quad \text{with $s=1-\frac{1}{p}$.}
\end{equation}
\end{cor}

In particular, when $p=2$, we get that
\[
B_{n}(f\et,\cS)\le C\cro{\pa{\frac{\sigma_{2}^{2}V(\frm+1)}{n}}^{\frac{1}{3}}+\sigma_{p}\sqrt{\frac{\frm+1}{n}}}.
\]

\begin{proof}
Let $z_{i}=\scal{\theta\et}{x_{i,n}}$ for $i\in \{1,\ldots,n\}$. Applying Proposition~\ref{prop-approxMono} with $k=1$ to the function $\phi\et$, which is monotone on $\R$, and to the points $\etc{z}$,  we obtain that for every $\gamma>0$ there exists a function $\overline \phi$ which is monotone and $K$-piecewise constant on $\R$, with $K=\PEI{\gamma+1}\le \gamma+1$, such that $n^{-1}\sum_{i=1}^{n}\ab{\phi\et(z_{i})-\overline \phi(z_{i})}\le  V/\gamma$. Since the function $\overline \phi(\scal{\theta\et}{\cdot})$ belongs to $\cO_{\text{s-pc}}(K)$, we obtain that
\begin{align*}
B_{n}(f\et,\cS)&\le \frac{3 V}{\gamma}+\kappa\sigma_{p}\pa{\frac{(\frm+1)K}{n}}^{s}.
\end{align*}

If $V=0$, we take $\gamma=1/2$, hence $K=1$, and deduce that $B_{n}(f\et,\cS)\le \kappa\sigma_{p}[(\frm+1)/n]^{s}$. Otherwise, we take $\gamma=(V\sigma_{p}^{-1}(n/(\frm+1))^{s})$ and \eref{eq-corr-single} follows from the inequalities
\begin{align*}
B_{n}(f\et,\cS)&\le \frac{3V}{\gamma}+\kappa\sigma_{p}\pa{\frac{(\frm+1)(\gamma+1)}{n}}^{s}\le  \frac{3V}{\gamma}+\kappa\sigma_{p}\pa{\frac{(\frm+1)\gamma}{n}}^{s}+\kappa\sigma_{p}\pa{\frac{\frm+1}{n}}^{s}.
\end{align*}
\end{proof}

\section{Bounding the degrees of some extremal functions}\label{sect-BD}
This section is devoted to the calculation of the degrees of some extremal functions.
Let $I$ be a nontrivial interval and $r$ a positive integer. We say that $f$ is {\em $r$-nondecreasing} on $I$ if for any $\gv=(v_{0},\ldots,v_{r})\in I^{r+1}$ with $v_{0}<\ldots<v_{r}$
\begin{equation}\label{eq-condphi1}
L_{\gv}(f)=\left|\begin{array}{ccccc}
1 & v_{0}&\ldots & v_{0}^{r-1}& f(v_{0}) \\
\vdots & \vdots&\ldots & \vdots& \vdots \\
1 & v_{r}&\ldots & v_{r}^{r-1}& {f(v_{r}) }\\
\end{array}\right|\ge 0.
\end{equation}
When $r=1$, $f$ is $r$-nondecreasing if for any points $v_{0}<v_{1}$ in $I$,
\[
L_{\gv}(f)=\left|\begin{array}{cc}
1 &  f(v_{0})  \\
1 & f(v_{1}) \\
\end{array}\right|=f(v_{1})-f(v_{0})\ge 0
\]
which simply means that $f$ is nondecreasing on $I$. A function $f$ is $2$-nondecreasing if for any points $v_{0}<v_{1}<v_{2}$ in $I$,
\begin{equation}\label{eq-condphi2}
L_{\gv}(f)=\left|\begin{array}{ccccc}
1 & v_{0}& f(v_{0}) \\
1 & v_{1}& f(v_{1}) \\
1 & v_{2}& {f(v_{2}) }\\
\end{array}\right|
=\pa{v_{2}-v_{0}}\cro{\frac{v_{2}-v_{1}}{v_{2}-v_{0}}f(v_{0})+\frac{v_{1}-v_{0}}{v_{2}-v_{0}}f(v_{2})-f(v_{1})}\ge 0,
\end{equation}
which means that $f$ is convex on $I$. Similarly, we shall say that $f$ is {\em $r$-nonincreasing} on $I$ if $-f$ is $r$-nondecreasing or equivalently if for any $\gv=(v_{0},\ldots,v_{r})\in I^{r+1}$ with $v_{0}<\ldots<v_{r}$
\begin{equation}\label{eq-condphi0}
L_{\gv}(f)=\left|\begin{array}{ccccc}
1 & v_{0}&\ldots & v_{0}^{r-1}& f(v_{0}) \\
\vdots & \vdots&\ldots & \vdots& \vdots \\
1 & v_{r}&\ldots & v_{r}^{r-1}& {f(v_{r}) }\\
\end{array}\right|\le 0.
\end{equation}
A function $f$ is {\em $r$-monotone} on $I$ if it is either $r$-nonincreasing or $r$-nondecreasing on this interval. We denote by $\cF_{r}(I)$ the class of $r$-monotone functions on $I$. The set $\cF_{1}(I)$ is therefore the set of monotone functions on $I$ while $\cF_{2}(I)$ is that of functions which are convex-concave on $I$. If $\overline f$ is a polynomial function on $I$ with degree not larger than $r-1$, then for any $\gv=(v_{0},\ldots,v_{r})\in I^{r+1}$,
\[
L_{\gv}(\overline f)=\left|\begin{array}{ccccc}
1 & v_{0}&\ldots & v_{0}^{r-1}& \overline f(v_{0}) \\
\vdots & \vdots&\ldots & \vdots& \vdots \\
1 & v_{r}&\ldots & v_{r}^{r-1}& {\overline f(v_{r}) }\\
\end{array}\right|=0
\]
and $\overline f$ is therefore both $r$-nonincreasing and $r$-nondecreasing. Since $L_{\gv}(f)=L_{\gv}(f-\overline f)$ for every function $f$ on $I$ and $\gv\in I^{r+1}$, when $f$ is $r$-nondecreasing (respectively $r$-nonincreasing) so is $f-\overline f$.

Given a nontrivial interval $\cX$ and two positive integers $r,k$, we denote by $\cF_{r,k}$ the class of functions $f$ on $\cX$ which are $r$-monotone on each element $I$ of a partition of $\cX$ into at most $k$ nontrivial intervals. The aim of this section is to prove the following result.
\begin{thm}\label{thm-VC}
Let $r$ be a positive integer, $\cF$ a subset of $\cF_{r,k}$ and $\overline f\in\cF$ a piecewise polynomial with degree at most $r-1$ based on a partition of $\cX$ into at most $K\ge 1$ nontrivial intervals. Then $\overline f$ is extremal in $\cF$ with degree at most $(r+1)k+r(K-1)$.
\end{thm}

Let us now turn to the proof of this result.
Hereafter, we write $i\equiv j$ when the integers $i,j$ share the same parity and $i\not \equiv j$ otherwise.
\begin{lem}\label{lem1-VC}
Let $r$ be a positive integer and $v_{0}<\ldots<v_{r}$ be $r+1$ points in a nontrivial interval $I$. There exists no $r$-nondecreasing function $f$ on $I$ that satisfies
\begin{equation}\label{eq-lem1-VC}
f(v_{i})>0\quad \text{for $i\equiv (r+1)$ while for $i\not\equiv (r+1)$}\quad f(v_{i})\le 0
\end{equation}
and there exists no $r$-nonincreasing function $f$ on $I$ that satisfies
\begin{equation}\label{eq-lem1-VCb}
f(v_{i})\le 0\quad \text{for $i\equiv (r+1)$ while for $i\not\equiv (r+1)$}\quad f(v_{i})> 0.
\end{equation}
\end{lem}

\begin{proof}
The result is a consequence of the fact that
\begin{align*}
L_{\gv}(f)&=\left|\begin{array}{ccccc}
1 & v_{0}&\ldots & v_{0}^{r-1}& f(v_{0}) \\
\vdots & \vdots&\ldots & \vdots& \vdots \\
1 & v_{r}&\ldots & v_{r}^{r-1}& f(v_{r}) \\
\end{array}\right|=\sum_{i=0}^{r}(-1)^{r-i}V_{i}f(v_{i})\\
&=\sum_{i=0}^{r}V_{i}f(v_{i})\1_{i\not \equiv r+1}-\sum_{i=0}^{r}V_{i}f(v_{i})\1_{i\equiv r+1}
\end{align*}
where $V_{0},\ldots,V_{r}$ are positive Vandermonde determinants.
\end{proof}

Note that it is sufficient to prove the Theorem~\ref{thm-VC} for the larger class $\cF_{r,k}$, which we shall do hereafter by taking $\cF=\cF_{r,k}$. Let $n> r(k+K-1)+k$ and $\cZ=\{z_{1},\ldots,z_{n}\}$ be an arbitrary subset of $\cX$ with cardinality $n$. We shall prove that there exists a subset $C$ of $\cZ$ such that
\begin{equation}\label{eq-thmVC00}
C\ne \{f-\overline f>0\}\cap \cZ\quad \text{for every $f\in\cF$}.
\end{equation}
This would imply that the VC-dimension of $\{\{f>\overline f\},\; f\in\cF\}$ is not larger than $r(k+K-1)+k$.

Using the fact that $\cF=\cF_{r,k}$ is stable with respect to the transformation $f\mapsto -f$ and applying this result to the function $-\overline f$, which satisfies the same assumptions as $\overline f$, we obtain that the VC-dimension $\{\{f<\overline f\},\; f\in\cF\}$ is not larger than $r(k+K-1)+k$ as well, hence that $\overline f$ is extremal in $\cF$ with degree not larger than $r(k+K-1)+k$.

With no loss of generality we may assume that $z_{1}<\ldots<z_{n}$. Given an interval $I\subset \cX$, we denote by $N(I)$ the number of these $z_{i}$ that belong to $I$.

Let $\cI$ denote the partition of $\cX$ with cardinality not larger than $K\ge 1$ for which the restriction of $\overline f$ to each interval $I\in\cI$ is a polynomial with degree not larger than $r-1$. Since $n\ge rK+r(k-1)+k\ge rK+1$, there exists at least one interval $I\in\cI$ which contains at least $r+1$ points. Let $I_{1},\ldots,I_{L}$ with $L\ge 1$ be the collection of these intervals which contain at least $r+1$ points. When $L\ge 2$, we assume that the labelling of these sets is such that the right endpoint of $I_{l}$ is not larger than the left endpoint of $I_{l+1}$ for $l\in\{1,\ldots,L-1\}$. Denoting by $\cZ(\overline f)$ the set $\bigcup_{l=1}^{L}\pa{\cZ\cap I_{l}}$,
we associate with the points $z\in \cZ(\overline f)$ a sign $\eps(z)\in\{-1,1\}$ which is defined as follows. Firstly, if $z=z_{i}\in I_{1}$ then $\eps(z_{i})=(-1)^{i}$. Then, if $L\ge 2$ and $l\in\{2,\ldots,L\}$ we define iteratively the signs of the points $z_{i}$ in $I_{l}$ by the formula $\eps(z_{i+1})=-\eps(z_{i})$ where the sign of the smallest point $z_{i}$ in $I_{l}$ is $(-1)^{r+1}$ times that of the largest point $z_{i}$ in $I_{l-1}$. This means that the sequence of signs associated with the elements $z_{i}$ in $I_{l}$ starts with the same sign as that of last point $z_{i}$ in $I_{l-1}$ when $r$ is odd or by the opposite sign when $r$ is even. Then the signs alternate along the sequence of $z_{i}$ for those $z_{i}\in I_{l}$. For example, when $L=3$, $N(I_{1})=4$, $N(I_{2})=N(I_{3})=3$ and the index $i$ of the smallest point $z_{i}$ in $I_{1}$ is even, we obtain that the sequences of signs that are associated with the points $z_{i}$ in $I_{1},I_{2}$ and $I_{3}$ are respectively given by $(+1,-1,+1,-1), (-(-1)^{r+1},(-1)^{r+1},-(-1)^{r+1}),(-1,1,-1)$. Finally, we define our set $C$ as
\[
C=\ac{z\in \cZ(\overline f),\; \eps(z)=1}.
\]

Let us assume by contradiction that there exists an element $f$ in $\cF$ for which equality holds in \eref{eq-thmVC00} that is
\begin{equation}\label{eq-thmVC01}
z\in\cZ(\overline f)\quad \text{and}\quad \eps(z)=+1\iff z\in\cZ(\overline f)\quad \text{and}\quad (f-\overline f)(z)>0.
\end{equation}
The function $f$ is associated with the partition $\cJ$ of $\cX$ into at most $k$ nontrivial intervals. The class of sets $\{I\cap J\ne \vide, (I,J)\in \cI\times \cJ\}$ is therefore a partition of $\cX$ into at most $k+K-1$ intervals. For $J\in \cJ$, let us denote by $D_{J}\ge 1$ the number of intervals $I\in\cI$ such that
$I\cap J\neq \vide$. Since $\sum_{J\in \cJ}D_{J}\le k+K-1$ and
\[
\sum_{J\in \cJ}\cro{N(J)-\pa{rD_{J}+1)}}\ge n-r\pa{k+K-1}-k>0,
\]
there exists at least one interval $J\in \cJ$ such that $N(J)\ge rD_{J}+2$. Two situations may then occur
\begin{itemize}
\item[(a)] there exists an interval $I_{l}$ with $l\in\{1,\ldots,L\}$ such that $I_{l}\cap J$ contains at least $r+2$ points;
\item[(b)] there exist two intervals $I_{l}$ and $I_{l'}$ with $l<l'$, $l,l'\in\{1,\ldots,L\}$, such that $I_{l}\cap J$ and $I_{l'}\cap J$ contain exactly $r+1$ points.
\end{itemize}
We shall prove that these two situations lead to a contradiction. Hereafter, we shall repeatedly use the following Property (P) that derives from \eref{eq-thmVC01}, Lemma~\ref{lem1-VC} and the fact $f-\overline f$ shares the same $r$-monotonicity as $f$.

{\bf Property (P)}: Let $I\cap J$, with $I\in\cI$ and $J\in\cJ$, be an interval which contains $r+1$ consecutive points $z_{m},\ldots, z_{m+r}$ for some $m\in\N$. If $f$ is $r$-nondecreasing on $I\cap J$, $\eps(z_{m})=(-1)^{r}$ while $\eps(z_{m})=(-1)^{r+1}$ when $f$ is $r$-nonincreasing on $I\cap J$.

As a consequence, it is impossible that $I\cap J$ contains more than $r+1$ points: if it contained $r+2$ consecutive points $z_{m},\ldots,z_{m+r+1}$, the Property (P) would be violated either by the sequence of $v_{0}=z_{m},\ldots,v_{r}=z_{m+r}$ or $v_{0}=z_{m+1},\ldots,v_{r}=z_{m+r+1}$. In particular, the case $(a)$ is impossible.

Let us now turn to the case $(b)$. If there exists $l''\in\{1,\ldots,L\}$ such that $l<l''<l'$, then $I_{l''}\subset J$ and $I_{l''}\cap J$ contains exactly $r+1$ points. As a consequence, with no loss of generality, we may assume that $l'=l+1$. Let $z_{m},\ldots,z_{m+r}$ be the $r+1$ elements of $\cZ \cap I_{l}\cap J$ and $z_{m'},\ldots,z_{m'+r}$ the $r+1$ elements of $\cZ\cap I_{l+1}\cap J$. We note that $z_{m},\ldots,z_{m+r}$ are necessarily the $r+1$ largest elements $z_{i}$ in $I_{l}$ while $z_{m'},\ldots,z_{m'+r}$ are the $r+1$ smallest elements $z_{i}$
in $I_{l+1}$. It follows from our construction that $\eps(z_{m'})=(-1)^{r+1}\eps(z_{m+r})$. If $f$ is $r$-nondecreasing on $J$, it is $r$-nondecreasing on both $I_{l}\cap J$ and $I_{l+1}\cap J$. It follows from Property (P) that $\eps(z_{m})=(-1)^{r}$, hence $\eps(z_{m+r})=+1$ and $\eps(z_{m'})=(-1)^{r+1}$. The property (P) is then violated on the interval $I_{l+1}\cap J$ with the points $z_{m'},\ldots,z_{m'+r}$. In the opposite, if $f$ is $r$-nondecreasing on $J$, we necessarily have $\eps(z_{m})=(-1)^{r+1}$, hence $\eps(z_{m+r})=-1$ and $\eps(z_{m'})=(-1)^{r}$ which is impossible by Property (P). The existence of a function $f$ in $\cF$ that satisfies \eref{eq-thmVC01} is therefore impossible. This concludes the proof of Theorem~\ref{thm-VC}.

\section{The proof of Theorem~\ref{thm-main1}}\label{sect-pfTh1}
Let $D\in\{1,\ldots,n\}$ be an integer such that $\cO(D)$ is not empty. For $f,g\in\cF$, we recall that $T(\bsZ,f,g)=\sum_{i=1}^{n}\pa{f(x_{i,n})-Y_{i,n}}\sgn(f,g,x_{i,n})$ and we set $\xi_{i,n}=Y_{i,n}-f\et(x_{i,n})=Y_{i,n}-\E\cro{Y_{i,n}}$ for all $i\in\{1,\ldots,n\}$, 
\begin{align*}
\Lambda(\gx,f,g)&=\E\cro{T(\bsZ,f,g)}=\sum_{i=1}^{n}\pa{f(x_{i,n})-f_{i}\et(x_{i,n})}\sgn(f,g,x_{i,n})\\
\widetilde{T}_{+}(\bsZ,f,g)&=T(\bsZ,f,g)-\E\cro{T(\bsZ,f,g)}=T(\bsZ,f,g)-\Lambda(\gx,f,g)\\
\widetilde{T}_{-}(\bsZ,f,g)&=\E\cro{T(\bsZ,f,g)}-T(\bsZ,f,g)=\Lambda(\gx,f,g)-T(\bsZ,f,g).
\end{align*}
With this notation, for $f,g\in\cF$   
\begin{align}
\widetilde{T}_{+}(\bsZ,f,g)&=\sum_{i=1}^{n}\pa{f(x_{i,n})-Y_{i,n}}\sgn(f,g,x_{i,n}) -\sum_{i=1}^{n}\pa{f(x_{i,n})-f_{i}\et(x_{i,n})}\sgn(f,g,x_{i,n})\nonumber\\
&=\sum_{i=1}^{n}\pa{f_{i}\et(x_{i,n})-Y_{i,n}}\sgn(f,g,x_{i,n})=\sum_{i=1}^{n}(-\xi_{i,n})\,\sgn(f,g,x_{i,n}).\label{eq-pfmain00}
\end{align}
By arguing similarly, we get 
\begin{align}
\widetilde{T}_{-}(\bsZ,f,g)&=\sum_{i=1}^{n}\xi_{i,n}\sgn(f,g,x_{i,n}).\label{eq-pfmain01}
\end{align}
It follows from the definitions of $\widetilde{T}_{+}(\bsZ,\cdot,\cdot )$ and $\widetilde{T}_{-}(\bsZ,\cdot,\cdot )$ and  Lemma~\ref{lem01} that for any $f \in\cF$ 
\begin{align*}
&\Lambda(\gx,f,\overline f)\\
&=\widetilde{T}_{-}(\bsZ,f,\overline f )+T(\bsZ,f,\overline f )\le \widetilde{T}_{-}(\bsZ,f,\overline f )+T(\bsZ,f)-T(\bsZ,\overline f)+T(\bsZ,\overline f)\\
&=\widetilde{T}_{-}(\bsZ,f,\overline f )+T(\bsZ,f)-T(\bsZ,\overline f)+\sup_{g\in\cF}T(\bsZ,\overline f,g)\\
&=\widetilde{T}_{-}(\bsZ,f,\overline f )+T(\bsZ,f)-T(\bsZ,\overline f)+\sup_{g\in\cF}\cro{\widetilde T_{+}(\bsZ,\overline f,g)+\E\cro{T(\bsZ,\overline f,g)}}\\
&\le \widetilde{T}_{-}(\bsZ,f,\overline f )+T(\bsZ,f)-T(\bsZ,\overline f) +\sup_{g\in\cF}\cro{\widetilde T_{+}(\bsZ,\overline f,g)+n\ell(f\et,\overline f)}.
\end{align*}
Consequently, for all $f\in\cF$
\begin{align}
\Lambda(\gx,f,\overline f)&\le n\ell(f\et,\overline f)+ T(\bsZ,f)-T(\bsZ,\overline f)+\sup_{g\in\cF}\widetilde{T}_{-}(\bsZ,g,\overline f )+\sup_{g\in\cF}\widetilde T_{+}(\bsZ,\overline f,g).\label{eq-pfmain12}
\end{align}
Applying Lemma~\ref{lem01} to $g=\overline f$ we deduce from \eref{eq-pfmain12} that  for any $f\in\cF$
\begin{align*}
n\ell(f,\overline f)&\le n\ell(f\et,\overline f)+\E\cro{\gT(\bsZ,f,\overline f)}=n\ell(f\et,\overline f)+\Lambda(\gx,f,\overline f)\\
&\le 2n\ell(f\et,\overline f)+T(\bsZ,f)-T(\bsZ,\overline f)+\sup_{g\in\cF}\widetilde{T}_{-}(\bsZ,g,\overline f )+\sup_{g\in\cF}\widetilde T_{+}(\bsZ,\overline f,g).
\end{align*}
In particular, this last inequality applies to $f=\widehat f\in\cE_{\frc}(\bsZ,\cF)$.  Since by definition of $\cE_{\frc}(\bsZ,\cF)$, $T(\bsZ,\widehat f)\le \inf_{g\in\cF}T(\bsZ,g)+\frc\le T(\bsZ,\overline f)+\frc$, we obtain that 
\begin{align}
n\ell(\widehat f,\overline f)&\le  2n\ell(f\et,\overline f)+\frc+\sup_{g\in\cF}\widetilde{T}_{-}(\bsZ,g,\overline f )+\sup_{g\in\cF}\widetilde T_{+}(\bsZ,\overline f,g).\label{eq-pfmain22}
\end{align}
By taking the expectation on both sides, we conclude by using the following lemma.
\begin{lem}\label{lem-T}
Let $D\in\{1,\ldots,n\}$ and $\overline f\in\cO(D)$. The exists a numerical constant $\kappa>0$ such that 
\begin{equation}\label{eq-lem-T02}
\E\cro{\sup_{g\in\cF}\widetilde{T}_{+}(\bsZ,\overline f,g)}\vee \E\cro{\sup_{g\in\cF}\widetilde{T}_{-}(\bsZ,g,\overline f)}\le  \frac{\kappa n}{2}\inf_{p\in [1,2]}\cro{\sigma_{p}(\gx)\pa{\frac{D}{n}}^{1-1/p}}.
\end{equation}
\end{lem}

\begin{proof}[Proof of Lemma~\ref{lem-T}]
We only prove the result for $\widetilde{T}_{+}$ since that for $\widetilde{T}_{-}$ can be established by arguing in the same manner.  Let us first note that for all $g\in\cF$,
\begin{align*}
\widetilde{T}_{+}(\bsZ,\overline f,g)&=-\sum_{i=1}^{n}\xi_{i,n}\sgn(\overline f,g,x_{i,n})=-\sum_{i=1}^{n}\xi_{i,n}\pa{\1_{\overline f>g}(x_{i,n})-\1_{\overline f< g}(x_{i,n})}\\
&\le \ab{\sum_{i=1}^{n}\xi_{i,n}\1_{\overline f< g}(x_{i,n})}+\ab{\sum_{i=1}^{n}\xi_{i,n}\1_{\overline f>g}(x_{i,n})}.
\end{align*}
Since $\overline f$ belongs to $\cO(D)$, the VC-dimensions of those of classes of subsets of $S=\{1,\ldots,n\}$ given by 
\[
\ac{\{i\in S,\; \overline f(x_{i,n})>g(x_{i,n})\}, g\in \cF}\quad \text{and}\quad \ac{\{i\in S,\; \overline f(x_{i,n})<g(x_{i,n})\}, g\in \cF}
\]
are not larger than $D$. It is therefore sufficient to show that if $\xi_{1,n},\ldots,\xi_{n,n}$ are independent centred random variables that satisfy
\[
\sigma_{p}^{p}=\frac{1}{n}\sum_{i=1}^{n}\E\cro{|\xi_{i,n}|^{p}}<+\infty\quad \text{for some $p\in [1,2]$}
\]
and if $\cI$ is a class of subsets of $S$ with dimension not larger than $d\in\{1,\ldots,n\}$, 
\begin{equation}\label{eq-cL1}
\E\cro{\sup_{I\subset \cI}\ab{\sum_{i\in I}\xi_{i,n}}}\le \frac{n\kappa}{4} \sigma_{p}\pa{\frac{d}{n}}^{1-1/p}.
\end{equation}
Let us first prove the result under the assumption that $\sigma_{2}<+\infty$. Hereafter $\nu_{n}$ denotes the uniform distribution on $S$ and $C,C'$ positive numerical constants that may vary from line to line. 
Since the $\xi_{i,n}$ are centred, by a classical symmetrisation argument (see Lemma 6.3 in Ledoux and Talagrand~\citeyearpar{MR2814399}) 
\begin{equation}\label{eq-sympf}
\E\cro{\sup_{I\in\cI}\ab{\sum_{i\in I}\xi_{i,n}}}\le 2\E\cro{\sup_{I\in\cI}\ab{\sum_{i\in I}\eps_{i}\xi_{i,n}}}=2\E\cro{\E_{\eps}\cro{\sup_{I\in\cI}\ab{\sum_{i\in I}\eps_{i}\xi_{i,n}}}}
\end{equation}
where $\etc{\eps}$ are  i.i.d.\ Rademacher random variables which are independent of the $\xi_{i,n}$. We denote by $\E_{\eps}$ the conditional expectation given $\xi_{1,n},\ldots,\xi_{n,n}$. Let us now fix the values $\xi_{1,n},\ldots,\xi_{n,n}$ assuming that they are not all equal to 0. We define the class of functions $\cG$ on $S$ by 
\[
g_{I}(i)=\frac{\xi_{i,n}}{\sqrt{n^{-1}\sum_{j=1}^{n}\xi_{j,n}^{2}}}\1_{I}(i)\quad \text{for $i\in S$ and $I\in\cI$.}
\]
We note that for every $I,I'\in\cI$,  $\norm{g_{I}}_{\L_{2}(\nu_{n})}^{2}=n^{-1}\sum_{i=1}^{n}g_{I}^{2}(i)\le 1$ and 
\begin{align*}
\frac{1}{n}\sum_{i=1}^{n}\ab{g_{I}(i)-g_{I'}(i)}^{2}=\sum_{i=1}^{n}\pa{\1_{I}(i)-\1_{I'}(i)}^{2}\frac{\xi_{i,n}^{2}}{\sum_{j=1}^{n}\xi_{j,n}^{2}}=\int_{S}\pa{\1_{I}-\1_{I'}}^{2}dQ
\end{align*}
where $Q$ is the probability measure on $S$ defined by $Q(i)=\xi_{i,n}^{2}/\sum_{j=1}^{n}\xi_{j,n}^{2}$ for all $i\in S$. We deduce that the diameter of $\cG$ with respect to the $\L_{2}(\nu_{n})$-distance is not larger than 1 and that the minimal number of $\L_{2}(\nu_{n})$-balls with radii $\eta\in (0,1)$ to cover $\cG$ is the same as the minimal number of $\L_{2}(Q)$-balls with the same radii that are necessary to  cover $\{\1_{I},\; I\in\cI\}$. By Theorem~2.6.4 in van der Vaart and Wellner~\citeyearpar{MR1385671}, we know that for every $\eta\in (0,1)$, this number is not larger than 
\[
N(\eta)\le C(d+1)\pa{\frac{4e}{\eta}}^{2d}\quad \text{for every $\eta>0$.}
\]
Setting 
\[
R_{n}(g_{I})=\frac{1}{n}\sum_{i=1}^{n}\eps_{i}g_{I}(i)=\frac{1}{n}\sum_{i\in I}\eps_{i}\frac{\xi_{i,n}}{\sqrt{n^{-1}\sum_{j=1}^{n}\xi_{j,n}^{2}}}\quad \text{and}\quad \norm{R_{n}}_{\cG}=\sup_{g\in\cG}\ab{R_{n}(g)}
\]
we derive from Theorem~3.11 in Koltchinskii~\citeyearpar{Koltchinski} that 
\[
\E_{\eps}\cro{\norm{R_{n}}_{\cG}}\le \frac{C}{\sqrt{n}}\int_{0}^{1}\sqrt{\log N(\eta)}d\eta\le C'\sqrt{\frac{d}{n}}.
\]
We deduce that 
\[
\E_{\eps}\cro{\sup_{I\in \cI}\ab{\sum_{i\in I}\eps_{i}\xi_{i,n}}}\le C\sqrt{nd}\pa{\frac{1}{n}\sum_{i=1}^{n}\xi_{i,n}^{2}}^{1/2}
\]
and note that this inequality remains true when the $\xi_{i,n}$ are all zero. By taking the expectation on both sides with respect to the $\xi_{i,n}$ and by using Jensen's inequality we get 
\begin{align*}
\E\cro{\sup_{I\in \cI}\ab{\sum_{i\in I}\eps_{i}\xi_{i,n}}}\le C\sqrt{nd}\pa{\frac{1}{n}\sum_{i=1}^{n}\E\cro{\xi_{i,n}^{2}}}^{1/2}=C\sigma_{2}\sqrt{nd}.
\end{align*}
We derive \eref{eq-cL1} from~\eref{eq-sympf} when $p=2$. Let us now turn to the case $p\in [1,2)$.

For $p=1$, we use the crude bound 
\begin{equation}\label{eq-pfLem-T00b}
\E\cro{\sup_{I\in\cI}\ab{\sum_{i\in\cI}\xi_{i,n}}}\le \E\cro{\sum_{i=1}^{n}\ab{\xi_{i,n}}}=n\sigma_{1}.
\end{equation}
For $p\in (1,2)$, we use a truncation argument as follows.  Let $\bs{\xi}'=(\xi_{1}',\ldots, \xi_{n}')$ be an independent copy of $\bs{\xi}=(\xi_{1},\ldots,\xi_{n})$. For every $a>0$, 
\begin{align*}
\E\cro{\sup_{I\in\cI}\ab{\sum_{i\in\cI}\xi_{i,n}}}&=\E\cro{\sup_{I\in\cI}\ab{\sum_{i\in\cI}\xi_{i,n}-\E\cro{\xi_{i,n}'}}}\le \E\cro{\sup_{I\in\cI}\ab{\sum_{i\in\cI}\pa{\xi_{i,n}-\xi_{i,n}'}}}\\
&\le \E\cro{\sup_{I\in\cI}\ab{\sum_{i\in\cI}\pa{\xi_{i,n}-\xi_{i,n}'}\1_{|\xi_{i,n}-\xi'_{i}|\le a}}}+\E\cro{\sup_{I\in\cI}\ab{\sum_{i\in\cI}\pa{\xi_{i,n}-\xi_{i,n}'}\1_{|\xi_{i,n}-\xi'_{i}|> a}}}.
\end{align*}
Using that  $\E\cro{|\xi_{i,n}-\xi_{i,n}'|^{p}}\le 2^{p}\E\cro{|\xi_{i,n}|^{p}}$ for all $i\in\{1,\ldots,n\}$,
\begin{align}
\E\cro{\sup_{I\in\cI}\ab{\sum_{i\in\cI}\pa{\xi_{i,n}-\xi_{i,n}'}\1_{|\xi_{i,n}-\xi'_{i}|> a}}}&\le \E\cro{\sum_{i=1}^{n}\ab{\xi_{i,n}-\xi_{i,n}'}\1_{|\xi_{i,n}-\xi_{i,n}'|>a}}\le\frac{n2^{p}\sigma_{p}^{p}}{a^{p-1}}.\label{eq-tronc1}
\end{align}
The random variables $(\xi_{1}-\xi_{1}')\1_{|\xi_{1}-\xi_{1}'|\le a},\ldots,(\xi_{n}-\xi_{n}')\1_{|\xi_{n}-\xi_{n}'|\le a}$ are centred and satisfy 
\begin{align*}
\sum_{i=1}^{n}\E\cro{|\xi_{i,n}-\xi_{i,n}'|^{2}\1_{|\xi_{i,n}-\xi_{i,n}'|\le a}}\le a^{2-p}\sum_{i=1}^{n}\E\cro{|\xi_{i,n}-\xi_{i,n}'|^{p}}\le a^{2-p}n2^{p}\sigma_{p}^{p}.
\end{align*}
Applying \eref{eq-cL1} with $p=2$ to these random variables we obtain that
\begin{align}
\E\cro{\sup_{I\in\cI}\ab{\sum_{i\in I}\pa{\xi_{i,n}-\xi_{i,n}'}\1_{|\xi_{i,n}-\xi_{i,n}'|\le a}}}\le  Ca^{1-p/2}2^{p/2}\sigma_{p}^{p/2}\sqrt{nd}.\label{eq-tronc2}
\end{align}
We deduce from \eref{eq-tronc1} and \eref{eq-tronc2} that for every $a>0$, 
\begin{align*}
&\E\cro{\sup_{I\in\cI}\ab{\sum_{i\in\cI}\xi_{i,n}}}\le \frac{n2^{p}\sigma_{p}^{p}}{a^{p-1}}+Ca^{1-p/2}2^{p/2}\sigma_{p}^{p/2}\sqrt{nd}
\end{align*}
and the result follows with by taking $a=(2\sigma_{p})(n/d)^{1/p}$.
\end{proof}

\section{The other proofs}\label{Sect-OP}

\subsection{Preliminary results}
\begin{lem}\label{lem-approxMoy}
Let $f$ be a function on a finite set $\cX$ with mean $\overline f=|\cX|^{-1}\sum_{x\in \cX}f(x)$. Then, 
\[
\frac{1}{|\cX|}\sum_{x\in \cX}\ab{f(x)-\overline f}\le \frac{V(f)}{2}\quad \text{with}\quad V(f)=\max_{x\in\cX}f(x)-\min_{x\in\cX}f(x).
\]
\end{lem}

\begin{proof}
Let $I_{+}=\{i\in\{1,\ldots,n\}, f(x_{i,n})\ge \overline f\}$ and $I_{-}=\{i\in\{1,\ldots,n\}, f(x_{i,n})<\overline f\}$. Using the convention $\sum_{\vide}=0$, 
\begin{align*}
\sum_{i=1}^{n}\ab{f(x_{i,n})-\overline f}&=\sum_{i\in I_{+}}\pa{f(x_{i,n})-\overline f}+\sum_{i\in I_{-}}\pa{\overline f-f(x_{i,n})}\\
&=\sum_{i\in I_{+}}f(x_{i,n})-\sum_{i\in I_{-}}f(x_{i,n})-\frac{|I_{+}|-|I_{-}|}{n}\pa{\sum_{i\in I_{+}}f(x_{i,n})+\sum_{i\in I_{-}}f(x_{i,n})}\\
&=\sum_{i\in I_{+}}f(x_{i,n})\pa{1-\frac{|I_{+}|-|I_{-}|}{n}}-\sum_{i\in I_{-}}f(x_{i,n})\pa{1+\frac{|I_{+}|-|I_{-}|}{n}}
\end{align*}
Setting $\lambda=|I_{+}|/n$, hence $1-\lambda=|I_{-}|/n$, $(|I_{+}|-|I_{-}|)/n=2\lambda-1$ and 
\begin{align*}
\frac{1}{n}\sum_{i=1}^{n}\ab{f(x_{i,n})-\overline f}&=\frac{2(1-\lambda)}{n}\sum_{i\in I_{+}}f(x_{i,n})-\frac{2\lambda}{n}\sum_{i\in I_{-}}f(x_{i,n})\\
&\le 2\lambda(1-\lambda)\max_{x\in\cX}f(x_{i,n})-2\lambda(1-\lambda)\min_{x\in\cX}f(x_{i,n})\\
&= 2\lambda(1-\lambda)V(f)\le \frac{V}{2}.
\end{align*}
\end{proof}

\begin{lem}\label{lem-approxM}
Let $f$ be a monotone function on the interval $\cX$. For each $K\ge 1$ there exists a partition of $\cX$ into at most $K\wedge n$ consecutive intervals and a function $\overline f$ which is constant on each of these intervals, shares the same monotonicity as $f$ and satisfies
\[
\ell(f,\overline f)\le \pa{\frac{1}{K}+\frac{1}{n}}\frac{\ab{f(x_{1,n})-f(x_{n,n})}}{2}\le \frac{\ab{f(x_{1,n})-f(x_{n,n})}}{K}.
\]
Besides, $\ell(f,\overline f)=0$ as soon as $K\ge n$.
\end{lem}

\begin{proof}
We may assume that $f$ is nondecreasing (otherwise it suffices to change $f$ into $-f$). Hence, $f(x_{1,n})\le f(x_{2})\ldots\le f(x_{n,n})$. If $K\ge n$, the function 
\[
\overline f=f(x_{1,n})\1_{(-\infty, x_{1,n}]}+\sum_{i=2}^{n-1}f(x_{i,n})\1_{(x_{i-1},x_{i,n}]}+f(x_{n,n})\1_{(x_{n-1},+\infty)\cap \cX}
\]
which satisfies $\overline f(x_{i,n})=f(x_{i,n})$ for all $i\in \{1,\ldots,n\}$ suits. Otherwise $K\in\{1,\ldots,n-1\}$ and there exist $q,r\in\N$ such that $n=Kq+r$ with $r\in\{0,\ldots,K-1\}$. We partition $\{1,\ldots,n\}$ as follows. For $k\in \{1,\ldots,K-r\}$ we set 
$I_{k}=\{(k-1)q+1,\ldots,kq\}$ and if $r\ne 0$ we complete the partition with $I_{K-r+l}=\{(K-r)q+(l-1)(q+1)+1,\ldots,(K-r)q+l(q+1)\}$ for $l\in\{1,\ldots,r\}$. This results in a partition $I_{1},\ldots,I_{K}$ of $\{1,\ldots,n\}$ such that $1\le q\le |I_{k}|\le q+1$ for all $k\in\{1,\ldots,K\}$. We may then find a partition of $\cX$ into $K$ intervals $\cX_{1},\ldots,\cX_{K}$ such that $\cX_{k}\cap\{x_{1,n},\ldots,x_{n,n}\}=\{x_{i,n},\; i\in I_{k}\}$ for $k\in\{1,\ldots,K\}$. For $k\in\{1,\ldots,K\}$, we set $\overline f_{k}=|I_{k}|^{-1}\sum_{i\in I_{k}}f(x_{i,n})$ and $\overline f=\sum_{k=1}^{K}\overline f_{k}\1_{\cX_{k}}$. Since $f$ is nondecreasing, $\overline f_{1}\le \ldots\le \overline f_{K}$ and $\overline f$ is therefore nondecreasing too. If follows from Lemma~\ref{lem-approxMoy} that 
\begin{align*}
\sum_{i=1}^{n}\ab{f(x_{i,n})-\overline f(x_{i,n})}=\sum_{k=1}^{K}\sum_{i\in I_{k}}\ab{f(x_{i,n})-\overline f_{k}}\le \frac{1}{2}\sum_{k=1}^{K}|I_{k}|\pa{\max_{i\in I_{k}}f(x_{i,n})-\min_{i\in I_{k}}f(x_{i,n})}.
\end{align*}
Using the inequalities $f(x_{1,n})\le \min_{i\in I_{1}}f(x_{i,n})\le \max_{i\in I_{K}}f(x_{i,n})\le f(x_{n,n})$ and for $K>1$, $\max_{i\in I_{k}}f(x_{i,n})\le \min_{i\in I_{k+1}}f(x_{i,n})$ for $k\in\{1,\ldots,K-1\}$, we obtain that 
\[
\sum_{i=1}^{n}\ab{f(x_{i,n})-\overline f(x_{i,n})}\le \frac{q+1}{2}\sum_{k=1}^{K}\pa{\max_{i\in I_{k}}f(x_{i,n})-\min_{i\in I_{k}}f(x_{i,n})}\le \frac{(q+1)(f(x_{n,n})-f(x_{1,n}))}{2}.
\]
Since  $q+1\le (n+K)/K$, we conclude that 
\[
\frac{1}{n}\sum_{i=1}^{n}\ab{f(x_{i,n})-\overline f(x_{i,n})}\le \pa{\frac{1}{K}+\frac{1}{n}}\frac{f(x_{n,n})-f(x_{1,n})}{2}.
\]
\end{proof}

\begin{lem}\label{lem-w}
Let $\mu$ be a Borel measure on $\R$ with modulus $w$ and $f$ a nonnegative concave function on $[a,b]$ bounded by $B>0$. Then 
\[
\int_{[a,b]}f(x)d\mu(x)\le Bw\pa{\frac{1}{B}\int_{a}^{b}f(x)dx}.
\]
\end{lem}
\begin{proof}
Since $f$ is concave and bounded by $B$,  the set $\{f>t\}$ for $t\in\R_{+}$ is a subinterval of $[a,b]$ which  is empty when $t\ge B$. Using that $f$ is nonnegative, we get  
\begin{align*}
\int_{[a,b]}f(x)d\mu(x)=\int_{0}^{+\infty}\mu\pa{\{f>t\}}dt=\int_{0}^{B}\mu\pa{\{f>t\}}dt\le \int_{0}^{B}w\pa{\lambda\pa{\{f>t\}}}dt.
\end{align*}
By applying Jensen's inequality, we obtain that  
\begin{align*}
B\cro{\frac{1}{B}\int_{0}^{B}w\pa{\{f>t\}}dt}\le Bw\pa{\frac{1}{B}\int_{0}^{B}\lambda\pa{\{f>t\}}dt}=Bw\pa{\frac{1}{B}\int_{a}^{b}f(x)dx}.
\end{align*}
\end{proof}

\begin{lem}\label{lem_approx_1-intervalle}
Let $[a,b]$ be a nontrivial compact interval, $\mu$ a Borel measure with modulus $w$ and $f$ a continuous, monotone and convex (or concave) function on $[a,b]$. We denote by $l$ its linear interpolant, which is defined by
\[
l:x\mapsto f(a)+\Delta(x-a)\quad \text{for $x\in [a,b]$ with}\quad \Delta=\frac{f(b)-f(a)}{b-a}.
\]
We set 
\begin{align*}
R_{1}&=\frac{\ab{f'_{l}(b)-f'_{r}(a)}(b-a)}{4}\mu((a,b));\\
R_{2}&=\ab{f(b)-f(a)}w\pa{\frac{1}{|f(b)-f(a)|}\int_{a}^{b}\ab{f(x)-l(x)}dx}
\end{align*}
with the convention $R_{2}=0$ when $f(a)=f(b)$, and 
\[
R_{3}=
\begin{cases}
\dps{\frac{\ab{f'_{r}(a)-\Delta}\ab{\Delta-f'_{l}(b)}}{\ab{f'_{r}(a)-f'_{l}(b)}}\pa{b-a}w\pa{\frac{b-a}{2}}}&\text{when $0<\ab{f'_{r}(a)-f'_{l}(b)}<+\infty$}\\
\dps{\ab{f'_{r}(a)-\Delta}\pa{b-a}w\pa{\frac{b-a}{2}}}&\text{when $\ab{f'_{l}(b)}=+\infty$}\\
\dps{\ab{f'_{l}(b)-\Delta}\pa{b-a}w\pa{\frac{b-a}{2}}}&\text{when $\ab{f'_{r}(a)}=+\infty$}\\
0&\text{when $f'_{r}(a)=f'_{l}(b)$.}
\end{cases}
\]
Then, 
\[
\int_{[a,b]}\ab{f(x)-l(x)}d\mu(x)\le \min\ac{R_{1},R_{2},R_{3}}.
\]
\end{lem}

\begin{proof}
Changing $f$ into $-f$ if ever necessary, it suffices to prove the result for a monotone and concave function $f$. The function $f-l$ is therefore nonnegative and concave. Let us start by establishing some uniform upper bounds $B$ on $f-l$. We distinguish between several cases. 

-- When $f'_{r}(a)<+\infty$, $f'_{r}(a)\ge \Delta$ and 
\begin{align*}
0\le f(x)-l(x)&\le f(a)+f'_{r}(a)(x-a)-l(x)=\pa{b-a}\cro{\pa{f'_{r}(a)-\Delta}\frac{x-a}{b-a}}
\end{align*}
hence, 
\[
\sup_{x\in [a,b]}\ab{f(x)-l(x)}=(b-a)\sup_{u\in [0,1]}\pa{f'_{r}(a)-\Delta}u=(b-a)|f'_{r}(a)-\Delta|.
\]

-- When $f'_{l}(b)<+\infty$, $f'_{l}(b)\le \Delta$ and 
\begin{align*}
0\le f(x)-l(x)&\le f(b)-f'_{l}(b)(b-x)-l(x)=\pa{b-a}\pa{\Delta-f'_{l}(b)}\cro{1-\frac{x-a}{b-a}}
\end{align*}
we therefore get 
\[
\sup_{x\in [a,b]}\ab{f(x)-l(x)}=(b-a)\sup_{u\in [0,1]}\pa{\Delta-f'_{l}(b)}\cro{1-u}=(b-a)|\Delta-f'_{l}(b)|.
\]

-- When $f'_{r}(a)$ and $f'_{l}(b)$ are both finite, a refined bound is obtained by using the fact that  for all $x\in [a,b]$
\begin{align*}
f(x)-l(x)&\le \min\ac{f(a)+f'_{r}(a)(x-a),f(b)-f'_{l}(b)(b-x)}-l(x)\\
&=\min\ac{\pa{f'_{r}(a)-\Delta}\pa{x-a},\pa{\Delta-f'_{l}(b)}\pa{b-x}}\\
&=\pa{b-a}\min\ac{\pa{f'_{r}(a)-\Delta}\frac{x-a}{b-a},\pa{\Delta-f'_{l}(b)}\pa{1-\frac{x-a}{b-a}}}
\end{align*}
which leads to 
\begin{align}
\sup_{x\in [a,b]}\pa{f(x)-l(x)}&\le \pa{b-a}\sup_{u\in [0,1]}\min\ac{\pa{f'_{r}(a)-\Delta}u,\pa{\Delta-f'_{l}(b)}(1-u)}\nonumber\\
&=\pa{b-a}\frac{\pa{f'_{r}(a)-\Delta}\pa{\Delta-f'_{l}(b)}}{f'_{r}(a)-f'_{l}(b)}\label{pf-eq-03}
\end{align}
when $f'_{r}(a)\ne f'_{l}(b)$. When $f'_{r}(a)=f'_{l}(b)$, $f=l$ on $[a,b]$ and $\sup_{x\in [a,b]}\pa{f(x)-l(x)}=0$. 

We may therefore choose $B$ as follows: 
\begin{equation}\label{pf-eq-005}
B=
\begin{cases}
(b-a)|f'_{r}(a)-\Delta| &\text{when $|f'_{l}(b)|=+\infty$}\\
(b-a)|\Delta-f'_{l}(b)|&\text{when $|f'_{r}(a)|=+\infty$}\\
\dps{\pa{b-a}\frac{\ab{f'_{r}(a)-\Delta}\ab{\Delta-f'_{l}(b)}}{f'_{r}(a)-f'_{l}(b)}}&\text{when $0<|f'_{l}(b)-f'_{r}(a)|<+\infty$}\\
0 & \text{when $f'_{l}(b)=f'_{r}(a)$}.
\end{cases}
\end{equation}
We also observe that when $0<|f'_{l}(b)-f'_{r}(a)|<+\infty$, $f'_{l}(b)\le \Delta\le f'_{r}(a)$ and
\begin{align}
&\pa{b-a}\frac{\pa{f'_{r}(a)-\Delta}\pa{\Delta-f'_{l}(b)}}{f'_{r}(a)-f'_{l}(b)}\nonumber\\
&=(b-a)\pa{f'_{r}(a)-f'_{l}(b)}\frac{\pa{f'_{r}(a)-\Delta}}{f'_{r}(a)-f'_{l}(b)}\pa{1-\frac{\pa{f'_{r}(a)-\Delta}}{f'_{r}(a)-f'_{l}(b)}}
\end{align}
and we may therefore choose 
\begin{equation}\label{pf-eq-001}
B=(b-a)\pa{f'_{r}(a)-f'_{l}(b)}\sup_{u\in [0,1]}u(1-u)=\frac{(b-a)\pa{f'_{r}(a)-f'_{l}(b)}}{4}.
\end{equation}
Note this bound also holds when $|f'_{l}(b)|+|f'_{r}(a)|=+\infty$ and $f'_{l}(b)=f'_{r}(a)$. 
Finally, since $f$ is monotone on $[a,b]$, for all $x\in [a,b]$
\begin{align*}
0\le f(x)-l(x)&\le \max\{f(a),f(b)\}-\min\{l(a),l(b)\}\\
&=\max\{f(a),f(b)\}-\min\{f(a),f(b)\}=\ab{f(a)-f(b)}
\end{align*}
and we may therefore take 
\begin{equation}\label{pf-eq-004}
B=\ab{f(a)-f(b)}.
\end{equation}

Let us now finish the proof of Lemma~\ref{lem_approx_1-intervalle}. Firstly, since $\ab{f(x)-l(x)}\le B$ for all $x\in [a,b)]$ and $\ab{f(x)-l(x)}=0$ for $x\in\{a,b\}$, 
\begin{equation}\label{pf-eq-002}
\int_{[a,b]}\ab{f(x)-l(x)}d\mu(x)=\int_{(a,b)}\ab{f(x)-l(x)}d\mu(x)\le B\mu((a,b)).
\end{equation}
The inequality $\int_{[a,b]}\ab{f(x)-l(x)}d\mu(x)\le  R_{1}$ follows thus from \eref{pf-eq-001}. 

Since $R_{1}=0$ if and only if $f=l$ on $[a,b]$, it suffices to show that $\int_{[a,b]}\ab{f(x)-l(x)}d\mu(x)\le \min\{R_{2},R_{3}\}$ when $f\ne l$, which implies that the values of $B$ given by \eref{pf-eq-005} and \eref{pf-eq-004} are necessarily positive. We may then apply Lemma~\ref{lem-w} and get that
\begin{equation}\label{pf-eq-003}
\int_{[a,b]}\ab{f(x)-l(x)}d\mu(x)=\int_{[a,b]}\pa{f(x)-l(x)}d\mu(x)\le Bw\pa{\frac{1}{B}\int_{a}^{b}\ab{f(x)-l(x)}dx}.
\end{equation}
This inequality applied to $B$ given by \eref{pf-eq-004} which leads to  $\int_{[a,b]}\ab{f(x)-l(x)}d\mu(x)\le R_{2}$. 

Let $c$ be an arbitrary point in $[a,b]$. When $f'_{r}(a)$ and $f'_{l}(b)$ are both finite, $f$ lies under the graphs of $x\mapsto f(a)+f'_{r}(a)(x-a)$ and $x\mapsto f(b)-f'_{l}(b)(b-x)$ and consequently, 
\begin{align*}
\int_{a}^{b}|f(x)-l(x)|dx&=\int_{a}^{b}\pa{f(x)-l(x)}dx=\int_{a}^{c}\pa{f(x)-l(x)}dx+\int_{c}^{b}\pa{f(x)-l(x)}dx\\
&\le \pa{f'_{r}(a)-\Delta}\int_{a}^{c}\pa{x-a}dx+\pa{\Delta-f'_{l}(b)}\int_{c}^{b}\pa{b-x}dx.
\end{align*}
Since $c$ is arbitrary, 
\begin{align*}
\int_{a}^{b}|f(x)-l(x)|dx&\le \inf_{c\in [a,b]}\cro{\pa{f'_{r}(a)-\Delta}\frac{(c-a)^{2}}{2}+\pa{\Delta-f'_{l}(b)}\frac{(b-c)^{2}}{2}}\\
&=\frac{(b-a)^{2}}{2}\frac{\pa{f'_{r}(a)-\Delta}\pa{\Delta-f'_{l}(b)}}{f'_{r}(a)-f'_{l}(b)}=\frac{(b-a)B}{2}
\end{align*}
where the value of $B$ is given by \eref{pf-eq-005} when  $0<|f'_{l}(b)-f'_{r}(a)|<+\infty$. This bound remains true for the other values of $B$ given by \eref{pf-eq-005} when $|f'_{r}(a)|$ or $|f'_{l}(b)|$ is infinite. It follows from \eref{pf-eq-003} and the fact that $w$ is nondecreasing that for such values of $B$
\[
\int_{a}^{b}|f(x)-l(x)|dx\le Bw\pa{\frac{1}{B}\int_{a}^{b}|f(x)-l(x)|dx}\le Bw\pa{\frac{b-a}{2}}=R_{3}.
\]
\end{proof}

We now introduce the following definition.
\begin{df}
Let $K$ be a positive integer and $f$ be a continuous, monotone and convex-concave function on the compact interval $[a,b]$, $a<b$. We say that $l$ is a $K$-linear interpolation of $f$ on $[a,b]$ if there exists a subdivision $a_{0}=a<a_{1}<\ldots<a_{K}=b$ such that $l$ is affine on $[a_{i-1},a_{i}]$ and coincide with $f$ at $a_{i-1}$ and $a_{i}$, i.e. $f(a_{i-1})=l(a_{i-1})$ and $f(a_{i})=l(a_{i})$,  for all $i\in\{1,\ldots,K\}$.
\end{df}
If $l$ is a linear interpolation of a continuous, monotone and convex-concave function $f$ on $[a,b]$, then $l$ is also continuous, monotone and convex-concave on this interval. In fact, $l$ has the same monotonicity as $f$ and if $f$ is convex (respectively concave) so is $l$.

\subsection{Proof of Theorem~\ref{thm-approx-con}}\label{sect-pf-thmapprox}
%
Let us assume that $f$ is nondecreasing and convex. When this condition is not satisfied we apply the result to 
\[
g(x)=
\begin{cases}
-f(x)&\text{for $x\in [a,b]$ when $f$ is nonincreasing and concave}\\
f(b-x)&\text{for $x\in [0,b-a]$ when $f$ is nonincreasing and convex}\\
-f(b-x)&\text{for $x\in [0,b-a]$ when $f$ is nondecreasing and concave.}\\
\end{cases}
\]
Under our assumption on $f$, $f'_{r}(a)\in [0,+\infty)$. Since the result is clear when $f'_{r}(a)=f'_{l}(b)$ (then $f$ is affine), we assume here after that $f'_{r}(a)\ne f'_{l}(b)$

Let $\delta$ be a number in the interval $(f'_{r}(a),f'_{l}(b))$. Since the function $\phi(x)=f(x)-x\delta$  is convex on $[a,b]$ and satisfies $\phi'_{r}(a)=f'_{r}(a)-\delta<0$ and $\phi'_{l}(b)=f'_{l}(b)-\delta>0$, the minimizer $c$ of $\phi$ belongs to $(a,b)$ and satisfies $\phi_{l}(c)=f'_{l}(c)-\delta\le 0$ and $\phi_{r}(c)=f'_{r}(c)-\delta\ge 0$. Hence, 
\begin{equation}\label{eq-deltac}
f'_{l}(c)\le \delta \le f'_{r}(c).
\end{equation}
It follows from Cauchy-Schwarz inequality that for  every subdivision $t_{0}=a<\ldots<t_{K'}=c$ of $[a,c]$, 
\begin{align}
\sum_{i=1}^{K'}\sqrt{(t_{i}-t_{i-1})\pa{f'_{l}(t_{i})-f'_{r}(t_{i-1})}}&\le \sqrt{(c-a)\sum_{i=1}^{m'}\pa{f'_{l}(t_{i})-f'_{r}(t_{i-1})}}\nonumber\\
&=\sqrt{(c-a)\pa{f'_{l}(c)+\sum_{i=1}^{m'-1}\pa{f'_{l}(t_{i})-f'_{r}(t_{i})}}}\nonumber\\
&\le \sqrt{\pa{f'_{l}(c)-f'_{r}(a)} (c-a)}.\label{eq-contrad}
\end{align}
We build a sequence of points $a_{0}<a_{1}<\ldots$ by induction as follows. Firstly, $a_{0}=a$ and for $i\ge 1$,  we set 
\[
I_{i}=\ac{x\in (a_{i-1},c],\ \sqrt{(x-a_{i-1})\pa{f'_{l}(x)-f'_{r}(a_{i-1})}}\le \frac{ \sqrt{\pa{f'_{l}(c)-f'_{r}(a)}(c-a)}}{K}}.
\]
Then, $a_{i}=\sup I_{i}$ if $a_{i-1}<c$ and $a_{i}=a_{i-1}=c$ otherwise. Since $x\mapsto f'_{l}(x)$ is nondecreasing and left-continuous, so is $x\mapsto \sqrt{(x-a_{i-1})\pa{f'_{l}(x)-f'_{r}(a_{i-1})}}$ and consequently, 
\[
\sqrt{(a_{i}-a_{i-1})\pa{f'_{l}(a_{i})-f'_{r}(a_{i-1})}}\le \frac{ \sqrt{\pa{f'_{l}(c)-f'_{r}(a)} (c-a)}}{K} \quad \text{for all $i\ge 1$.}
\]
Besides if for some $i\ge 1$, $a_{i}<c$  then for any $a_{i}'\in (a_{i},c]$ 
\[
\sqrt{(a_{i}'-a_{i-1})\pa{f'_{l}(a_{i}')-f'_{r}(a_{i-1})}}> \frac{ \sqrt{\pa{f'_{l}(c)-f'_{r}(a) } (c-a)}}{K}.
\]
In particular, if  there exists an integer $k\ge 1$ such that $a=a_{0}<\ldots<a_{k}<c$, by setting 
$a'_{k}=c$ and using the fact that the mapping $x\mapsto f'_{r}(x)$ is right-continuous, we can find $a_{k-1}'\in (a_{k-1},a'_{k})$ such that 
\[
\sqrt{(a_{k}'-a_{k-1}')\pa{f'_{l}(a_{k}')-f'_{r}(a_{k-1}')}}> \frac{ \sqrt{\pa{f'_{l}(c)-f'_{r}(a) } (c-a)}}{K}.
\]
By induction, we obtain a sequence $c=a_{k}'>\ldots>a'_{0}>a$ such that
\[
\sqrt{(a_{i}'-a_{i-1}')\pa{f'_{l}(a_{i}')-f'_{r}(a_{i-1}')}}> \frac{ \sqrt{\pa{f'_{l}(c)-f'_{r}(a) } (c-a)}}{m}\quad \text{for all $i\in\{1,\ldots,k\}$}.
\]
Since the subdivision $a_{0}=a<a_{0}'<\ldots<a'_{k}=c$ of $[a,c]$ satisfies 
\begin{align*}
&\sqrt{(a_{0}'-a)f'_{l}(a_{0}')}+\sum_{i=1}^{k}\sqrt{(a_{i}'-a_{i-1}')\pa{f'_{l}(a_{i}')-f'_{r}(a_{i-1}')}}> \frac{ k\sqrt{\pa{f'_{l}(c)-f'_{r}(a) } (c-a)}}{K},
\end{align*}
in view of \eref{eq-contrad} the number $k$ must be smaller than $K$. We have therefore obtained a subdivision $a_{0}=a<\ldots<a_{K'}=c$ of $[a,c]$ with $K'\le K$ that satisfies
\[
(a_{i}-a_{i-1})\pa{f'_{l}(a_{i})-f'_{r}(a_{i-1})}\le \frac{ \pa{f'_{l}(c)-f'_{r}(a) }(c-a)}{K^{2}} \quad \text{for all $i\in\{1,\ldots,K'\}$.}
\]
Adding $K-K'$ additional points to this subdivision, we may assume hereafter wihout loss of generality that $K'=K$. Let $\overline f_{1}$ be the $K$-linear interpolation of $f$ based on this subdivision. Applying Lemma~\ref{lem_approx_1-intervalle} (the bound with $R_{1}$) we obtain that 
\begin{align}
\int_{[a,c]}\ab{f-\overline f_{1}}dQ&\le \sum_{i=1}^{K}\int_{[a_{i-1},a_{i}]}\ab{f-\overline f_{1}}dQ
\le \frac{\pa{f'_{l}(c)-f'_{r}(a) }(c-a)}{4K^{2}}\sum_{i=1}^{K}Q\pa{(a_{i-1},a_{i})}\nonumber\\
&\le \frac{\pa{f'_{l}(c)-f'_{r}(a) })(c-a)}{4K^{2}}\le \frac{\pa{f'_{l}(c)-f'_{r}(a) })(b-a)}{4K^{2}}.\label{eq-linter1}
\end{align}
We note that when $f'_{l}(b)<+\infty$, we may take $c=b$ in which case inequality \eref{eq-linter1} leads to \eref{eq-approx-con00}.

When $c<b$, we partition $[c,b]$ as follows. Since $f$ is nondecreasing, continuous and satisfies $f(c)\ne f(b)$, there exists a subdivision $b_{0}=c<\ldots<b_{K}=b$ of $[c,b]$ such that  
\[
f\pa{b_{i}}-f(b_{i-1})=\frac{f(b)-f(c)}{K}>0\quad \text{for all $i\in\{1,\ldots,K\}$.}
\]
We denote by $\overline f_{2}$ the $K$-linear interpolation of $f$ on $[c,b]$ which is based on this subdivision and set 
\[
\Delta_{i}=\frac{f(b_{i})-f(b_{i-1})}{b_{i}-b_{i-1}}=\frac{f(b)-f(c)}{K\pa{b_{i}-b_{i-1}}}>0\quad \text{for all $i\in\{1,\ldots,K\}$.}
\]
Applying Lemma~\ref{lem_approx_1-intervalle} (the bound with $R_{3}$), we obtain that for all $i\in\{1,\ldots,K\}$,  
\begin{align*}
\int_{[b_{i-1},b_{i}]}\ab{f-\overline f_{2}}dQ\le \overline R_{i}=\frac{\pa{\Delta_{i}-f'_{r}(b_{i-1})}\pa{f'_{l}(b_{i})-\Delta_{i}}}{f'_{l}(b_{i})-f'_{r}(b_{i-1})}\pa{b_{i}-b_{i-1}}w\pa{\frac{b_{i}-b_{i-1}}{2}}
\end{align*}
with the conventions that $\overline R_{i}=0$ when $f'_{r}(b_{i-1})=f'_{l}(b_{i})$ and 
\[
\overline R_{K}=\pa{\Delta_{K}-f'_{r}(b_{K-1})}\pa{b-b_{K-1}}w\pa{\frac{b-b_{K-1}}{2}}\quad \text{when $f'_{l}(b)=+\infty$.}
\]
It follows from the definition of $c$ and the convexity of $f$ that $f'_{l}(u)$ and $f'_{r}(u)$ are positive for all $u\in [c,b)$, hence when $\overline R_{i}>0$, $i\in\{1,\ldots,K\}$, we may write that
\begin{align*}
B_{i}&=\frac{\pa{\Delta_{i}-f'_{r}(b_{i-1})}\pa{f'_{l}(b_{i})-\Delta_{i}}}{f'_{l}(b_{i})-f'_{r}(b_{i-1})}
=\Delta_{i}^{2}\frac{\pa{\Delta_{i}^{-1}-(f'_{l}(b_{i}))^{-1}}\pa{\Delta_{i}^{-1}-(f'_{r}(b_{i-1}))^{-1}}}{(f'_{r}(b_{i-1}))^{-1}-(f'_{l}(b_{i}))^{-1}}\\
&\le \Delta_{i}^{2}\cro{\pa{\frac{1}{\Delta_{i}}-\frac{1}{f'_{l}(b_{i})}}\wedge \pa{\frac{1}{f'_{r}(b_{i-1})}-\frac{1}{\Delta_{i}}}}\le  \Delta_{i}^{2}\cro{\frac{1}{\Delta_{i}}\wedge \pa{\frac{1}{f'_{r}(b_{i-1})}-\frac{1}{\Delta_{i}}}}
\end{align*}
with the convention $1/f'_{l}(b)=0$ when $f'_{l}(b)=+\infty$. Setting $\psi(u)=\overline w(u)/u$ for all $u>0$, we obtain that for all $i\in\{1,\ldots,K\}$
\begin{align*}
\overline R_{i}&\le B_{i}\pa{b_{i}-b_{i-1}}w\pa{\frac{b_{i}-b_{i-1}}{2}}=B_{i}(b_{i}-b_{i-1})w(0)+\frac{B_{i}\pa{b_{i}-b_{i-1}}^{2}}{2}\psi\pa{\frac{b_{i}-b_{i-1}}{2}}\\
&\le \Delta_{i}(b_{i}-b_{i-1})w(0)+\frac{\Delta_{i}^{2}\pa{b_{i}-b_{i-1}}^{2}}{2}\psi\pa{\frac{b_{i}-b_{i-1}}{2}} \pa{\frac{1}{f'_{r}(b_{i-1})}-\frac{1}{\Delta_{i}}}\\
&=\frac{f(b)-f(c)}{K}w(0)+\frac{\pa{f(b)-f(c)}^{2}}{2K^{2}}\psi\pa{\frac{b_{i}-b_{i-1}}{2}} \pa{\frac{1}{f'_{r}(b_{i-1})}-\frac{1}{\Delta_{i}}}\\
&=\frac{f(b)-f(c)}{K}w(0)+\frac{\pa{f(b)-f(c)}^{2}}{2K^{2}}\psi\pa{\frac{f(b)-f(c)}{2K\Delta_{i}}} \pa{\frac{1}{f'_{r}(b_{i-1})}-\frac{1}{\Delta_{i}}}.
\end{align*}
Since $w$ is concave, $\psi$ is nonincreasing which implies that for all $i\in\{1,\ldots,K\}$
\begin{align*}
\overline R_{i}&\le\frac{f(b)-f(c)}{K}w(0)+\frac{\pa{f(b)-f(c)}^{2}}{2K^{2}}\int_{1/\Delta_{i}}^{1/f'_{r}(b_{i-1})}\psi\pa{\frac{(f(b)-f(c))t}{2K}} dt\\
&\le \frac{f(b)-f(c)}{K}w(0)+\frac{\pa{f(b)-f(c)}^{2}}{2K^{2}}\int_{1/f'_{l}(b_{i})}^{1/f'_{r}(b_{i-1})}\psi\pa{\frac{(f(b)-f(c))t}{2K}} dt
\end{align*}
where we have used the fact that $ 1/\Delta_{i}\ge 1/f'_{l}(b_{i})$ for $i\in\{1,\ldots,K\}$.
Since   $1/f'_{l}(b_{i})\ge 1/f'_{r}(b_{i})$ for $i\in\{1,\ldots,K-1\}$, by summing over $i\in\{1,\ldots,K\}$ we get 
\begin{align}
\int_{[c,b]}\ab{f-\overline f_{2}}dQ&\le \pa{f(b)-f(c)}w(0)+\frac{\pa{f(b)-f(c)}^{2}}{2K^{2}}\int_{1/f'_{l}(b)}^{1/f'_{r}(c)}\psi\pa{\frac{(f(b)-f(c))t}{2K}} dt\nonumber\\
&\le \pa{f(b)-f(a)}w(0)+\frac{\pa{f(b)-f(a)}^{2}}{2K^{2}}\int_{1/f'_{l}(b)}^{1/f'_{r}(c)}\psi\pa{\frac{(f(b)-f(a))t}{2K}} dt\nonumber\\
&=Vw(0)+\frac{V}{K}\int_{V/[2K f'_{l}(b)]}^{V/[2K f'_{r}(c)]}\psi\pa{s} ds\label{eq-linter2}.
\end{align}
When $f'_{r}(a)>0$, we may take $c=a$ in which case this bound would remains true.
Using  \eref{eq-deltac}, \eref{eq-linter1}and \eref{eq-linter2} we obtain that for every $\delta\in [f'_{r}(a),f'_{l}(b)]\cap(0,+\infty)$, there exists a ($2K$)-linear interpolation $\overline f=\overline f_{1}\1_{[a,c]}+\overline f_{2}\1_{[c,b]}$ of $[a,b]$  (depending on $\delta$)  such that  
\begin{align*}
\int_{[a,b]}\ab{f-\overline f}dQ
&=Vw(0)+ \frac{\pa{\delta-f'_{r}(a)}(b-a)}{4K^{2}} +\frac{V}{K}\cro{\Psi\pa{\frac{V}{2K\delta}}-\Psi\pa{\frac{V}{2Kf'_{l}(b)}}},
\end{align*}
which is \eref{eq-approx-con01}.

Let us now assume  $w(u)=w(0)+A[u/(b-a)]^{\alpha}$ with $A\ge 1$ and $\alpha\in [0,1]$.  
If $V=0$, $f$ is constant on $[a,b]$ and the inequality is satisfied for $\overline f=f$. Let us now assume that $V>0$. When $\alpha>0$, $\Psi(x)=A(x/[\alpha(b-a)])^{\alpha}$ and we set 
\[
\delta\et=\frac{V}{b-a}\pa{\frac{2^{1-\alpha}A}{\alpha}K^{1-\alpha}}^{1/(1+\alpha)}\ge \frac{V}{b-a}\ge \underline \delta\ge 0. 
\]
When $\delta\et\le \overline \delta$, we we may choose $\delta=\delta\et>0$ and get  
\begin{align}
\int_{[a,b]}\ab{f-\overline f}dQ&\le Vw(0)+ \frac{\pa{\delta-\underline \delta}(b-a)}{4K^{2}}+\frac{V}{K}\cro{\Psi\pa{\frac{V}{2K\delta}}-\Psi\pa{\frac{V}{2K\overline \delta}}}\nonumber\\
&\le Vw(0)+ \frac{\pa{\delta-\underline \delta}(b-a)}{2K^{2}}+\frac{V}{K}\times \frac{A V^{\alpha}}{\alpha (b-a)^{\alpha}2^{\alpha}K^{\alpha}}\pa{\frac{1}{\delta^{\alpha}}-\frac{1}{\overline \delta^{\alpha}}}\nonumber\\
&=Vw(0)+2V\pa{\frac{A}{\alpha 4^{\alpha}K^{1+3\alpha}}}^{\frac{1}{1+\alpha}}-\frac{\underline \delta(b-a)}{2K^{2}}-\frac{A V^{\alpha+1}}{\alpha (b-a)^{\alpha}2^{\alpha}K^{1+\alpha}}\frac{1}{\overline \delta^{\alpha}}\label{eq-refined}\\
&\le Vw(0)+2V\pa{\frac{A}{\alpha 4^{\alpha}K^{1+3\alpha}}}^{\frac{1}{1+\alpha}}= Vw(0)+V\pa{\frac{A2^{1-\alpha}}{\alpha K^{1+3\alpha}}}^{\frac{1}{1+\alpha}}.\label{eq-crude}
\end{align}
When $\delta\et> \overline \delta$, then $ \overline \delta<+\infty$ and we may take $\delta=\overline \delta$ which leads to 
\[
\int_{[a,b]}\ab{f-\overline f}dQ\le Vw(0)+ \frac{\pa{\overline \delta-\underline \delta}(b-a)}{4K^{2}}\le Vw(0)+V\pa{\frac{A2^{1-\alpha}}{\alpha K^{1+3\alpha}}}^{\frac{1}{1+\alpha}}.
\]
This complete the proof of \eref{eq-approx-con}. When $\alpha=1=A$, $\delta\et=V/(b-a) \in [\underline \delta,\overline \delta]\cap (0,+\infty)$ and \eref{eq-approx-con2} derives from~\eref{eq-refined}.

\subsection{Proof of Proposition~\ref{prop-approxMono}}\label{sect-pf-prop-approxMono}
There exists a partition $\cJ$ of $\cX$ into at most $k$ nontrivial intervals $J$ on which $f$ is monotone. We distinguish between two cases. 

-- Case $V_{\cJ}=V_{\cJ}(\gx,f)=0$. Then  $n_{J}V_{J}=0$ for all $J\in\cJ$ and $f$ is therefore constant on each element $J$ for which $n_{J}>0$. The function $\overline f$ which takes the value $\inf_{x\in J}f(x)$ on $J$ belongs to $\cO_{\text{pc}}(K)$ with $K\le k$ and it satisfies $\ell(f,\overline f)=0$, hence~\eref{eq-approxcm}.

-- Case $V_{\cJ}>0$. For $J\in\cJ$, let us define
\[
K_{J}=\PES{\gamma\sqrt{\frac{n_{J}}{n}\frac{V_{J}}{V_{\cJ}}}}\le \gamma\sqrt{\frac{n_{J}}{n}\frac{V_{J}}{V_{\cJ}}}+1\quad \text{and}\quad K_{J}\ge \gamma\sqrt{\frac{n_{J}}{n}\frac{V_{J}}{V_{\cJ}}}\vee 1.
\]
Since $|\cJ|\le k$, it follows from the definition \eref{def-variation} of $V_{\cJ}$ that 
\[
\sum_{J\in\cJ}K_{J}\le k+\gamma\sum_{J\in\cJ}\sqrt{\frac{n_{J}}{n}\frac{V_{J}}{V_{\cJ}}}=k+\gamma.
\]

By Lemma~\ref{lem-approxM}, for each element $J\in\cJ$ there exists a function $\overline f_{J}$ with the same monotonicity as $f$ on $J$, which is piecewise constant on a partition of $J$ into at most $K_{J}$ intervals and satisfies   
\[
\sum_{i=1}^{n}\ab{f(x_{i,n})-\overline f_{J}(x_{i,n})}\1_{J}(x_{i,n})\le \frac{n_{J}V_{J}}{K_{J}}.
\]
The function $\overline f=\sum_{J\in\cJ}\overline f_{J}\1_{J}$ therefore belongs to $\cO_{\text{pc}}(K)$ with 
$K= \sum_{J\in\cJ}K_{J}\le k+\gamma$ and it satisfies 
\[
\frac{1}{n}\sum_{i=1}^{n}\ab{f(x_{i,n})-\overline f(x_{i,n})}\le  \sum_{J\in \cJ}\frac{n_{J}V_{J}}{nK_{J}}\le \frac{\sqrt{V_{\cJ}}}{\gamma}\sum_{J\in \cJ}\sqrt{\frac{n_{J}V_{J}}{n}}= \frac{V_{\cJ}}{\gamma},
\]
which is~\eref{eq-approxcm}.

\subsection{Proof of Proposition~\ref{prop-approx-mcc}}\label{pf-prop-mcc}
There exists a partition $\cJ$ of $\cX$ into $k$  intervals (of positive lengths) such that for each $J\in\cJ$, $f$ is continuous, monotone convex-concave on the closure $[u_{J},v_{J}]$ of $J$. For $J\in\cJ$, let $W_{J}=(n_{J}V_{J}\Gamma_{J}/n)^{1/3}$.

If $W=0$ then $f$ is affine on all the intervals $J\in\cJ$, $f$ therefore belongs to $\cO_{\text{pl}}(k)$ and \eref{eq-approx-mcc} is satisfied with $\overline f=f\et$. 

Let us now consider the situation where $W>0$ and fix some $J\in\cJ$.

-- If $n_{J}=0$, we take $l_{J}$ the affine function on $[u_{J},v_{J}]$ that joints $(u_{J},f(u_{J}))$ to $(v_{J},f(v_{J}))$. Then $l_{J}$ is a 1-interpolation of $f$ on $[u_{J},v_{J}]$. 

-- If $n_{J}\ge 1$, we denote by $Q_{J}$ the probability on $J$ given by $n_{J}^{-1}\sum_{x_{i,n}\in J}\delta_{x_{i,n}}$. Since the design is regular, $\lambda(J)\le (n_{J}+1)L/n$ and for every interval $J'\subset J$, $(n_{J}Q_{J}(J')-1)L/n\le \lambda(J')$, that is $Q_{J}(J')\le 1/n_{J}+n\lambda(J')/(Ln_{J})$, hence 
\[
Q(J')\le \frac{1}{n_{J}}+\pa{1+\frac{1}{n_{J}}}\frac{n\lambda(J')}{(n_{J}+1)L}\le \frac{1}{n_{J}}+\pa{1+\frac{1}{n_{J}}}\frac{\lambda(J')}{\lambda(J)}\le w_{J}(\lambda(J'))
\]
with $w_{J}(u)=2/n_{J}+u/\lambda(J)$ for all $u\ge 0$. Applying Theorem~\ref{thm-approx-con} (more precisely \eref{eq-approx-con2}) with 
\[
K_{J}=\PES{\frac{\gamma W_{J}}{W^{1/3}}}\le\frac{\gamma W_{J}}{W^{1/3}}+1\quad \text{so that}\quad K_{J}\ge \frac{\gamma W_{J}}{W ^{1/3}}\vee 1
\]
we built a $(2K_{J})$-linear interpolation $\overline f_{J}$ of $f$ on $[u_{J},v_{J}]$ such that 
\begin{equation}\label{eq-approxmcc00}
\int_{J}\ab{f-\overline f_{J}}dQ_{J}\le \frac{n}{n_{J}}\pa{\frac{2V_{J}}{n}+\frac{n_{J}V_{J}\Gamma_{J}}{nK_{J}^{2}}}\le \frac{n}{n_{J}}\pa{\frac{2V_{J}}{n}+\frac{W_{J}W^{2/3}}{\gamma^{2}}}.
\end{equation}
We define $\overline f=\sum_{J\in\cJ}\overline f_{J}\1_{J}$. It is a $K$-linear interpolation of $f$ with 
\begin{align*}
K&\le \sum_{J\in\cJ}\pa{\1_{n_{J}=0}+2K_{J}\1_{n_{j}\ge 1}}\le \sum_{J\in\cJ}\cro{\1_{n_{J}=0}+2\pa{\frac{\gamma W_{J}}{W^{1/3}}+1}\1_{n_{j}\ge 1}}\\
&\le 2\pa{|\cJ|+\gamma}=2(k+\gamma)
\end{align*}
and by \eref{eq-approxmcc00} it satisfies 
\begin{align*}
\frac{1}{n}\sum_{i=1}^{n}\ab{f(x_{i,n})-\overline f(x_{i,n})}&\le \sum_{J\in\cJ}\cro{\frac{n_{J}}{n}\int_{J}\ab{f-\overline f}dQ_{J}}\1_{n_{J}\ge 1}\\
&\le \sum_{J\in\cJ}\pa{\frac{2V_{J}}{n}+\frac{W_{J}W^{2/3}}{\gamma^{2}}}=\frac{2V}{n}+\frac{W}{\gamma^{2}}.
\end{align*}

\subsection{Proof of Proposition~\ref{prop-VCSingle}}
Since the function $\overline \phi$ is constant on each element of the partition $\cJ$, the class $\{R_{J}=\{x\in\R^{\frm},\; \scal{x}{\overline \theta}\in J\},\; J\in\cJ\}$ partition $\R^{\frm}$ into at most $|\cJ|\le K$ regions which are delimited by parallel hyperplans. On each set $R_{J}$, $\overline f$ is constant and we denote by $c_{J}$ the value of this constant. Let $z_{1},\ldots,z_{n}$ be $n$ points in $\cX=\R^{\frm}$ with $n\ge (\frm+1)K+1$. There exists at least one region $R_{J\et}$ which contains at least $\frm+2$ points. Let $\cZ(J\et)$ be the set $\{z_{1},\ldots,z_{n}\}\cap R_{J\et}$. The class $\cA_{>}(\overline f)$ consists of the sets of form
\[
\ac{x\in\R^{\frm},\; \phi\pa{\scal{x}{\theta}}> \overline \phi\pa{\scal{x}{\overline \theta}}}
\]
where $\theta$ runs among $\sS_{\frm}$ and $\phi$ among the class $\cM$ of monotone functions on $\R$. Let us now assume by contradiction that $\cA_{>}(\overline f)$ could shatter $\cZ$. Then,  it could in particular shatter $\cZ(J\et)\subset \cZ$. Since $f\et$ takes the value $c_{J\et}$ on $\cZ(J\et)$, sets of the form
\[
\cH(\phi,\theta)=\ac{x\in\R^{\frm},\; \phi\pa{\scal{x}{\theta}}>c_{J\et}}=\ac{x\in\R^{\frm},\; \scal{x}{\theta}\in \phi^{-1}\pa{(c_{J\et},+\infty)}}
\]
with $\phi\in \cM$ and $\theta\in \sS_{\frm}$ could therefore shatter $\cZ(J\et)$. However, for  $\phi\in\cM$ and $\theta\in \sS_{\frm}$, $\cH(\phi,\theta)$ is either empty or a half-space and we know that the class of half-spaces cannot shatter more than $\frm+1$ points (see Pollard~\citeyearpar{MR762984} [Lemma 18]). This leads to a contradiction and proves that the VC-dimension of $\cA_{>}(\overline f)$ is not larger than $(\frm+1)K$. The result for the class $\cA_{<}(\overline f)$ can be obtained by arguing in the same manner.

\bibliographystyle{apalike}

\end{document}